\documentclass[a4paper,reqno,11pt]{amsart}

\usepackage[utf8]{inputenc}
\usepackage{microtype}

\usepackage[margin=1.2in]{geometry}

\usepackage{amssymb}
\usepackage{mathrsfs}
\usepackage{mathtools}
\usepackage{enumitem}
\usepackage{longtable}

\usepackage{color}
\usepackage{tikz}
\usetikzlibrary{topaths}

\usepackage{savesym}
\usepackage[all]{xy}
\savesymbol{cir}

\usepackage{hyperref}

\newtheorem{theorem}{Theorem}[section]
\newtheorem*{theorem*}{Theorem}
\newtheorem{lemma}[theorem]{Lemma}
\newtheorem{proposition}[theorem]{Proposition}
\newtheorem{corollary}[theorem]{Corollary}

\newtheorem{question}[theorem]{Question}
\newtheorem{cit}[theorem]{Citation}
\newtheorem{observation}[theorem]{Observation}

\newtheorem*{main:alt}{Theorem~\ref{thrm:alt}}
\newtheorem*{main:BV_alt}{Corollary~\ref{cor:BV_alt}}
\newtheorem*{main:BNS_BF}{Theorem~\ref{thrm:BNS_BF}}

\theoremstyle{definition}
\newtheorem{definition}[theorem]{Definition}
\newtheorem{remark}[theorem]{Remark}
\newtheorem{example}[theorem]{Example}

\newcommand{\Z}{\mathbb{Z}}
\newcommand{\N}{\mathbb{N}}
\newcommand{\R}{\mathbb{R}}
\newcommand{\symm}{S}
\newcommand{\I}{^{-1}}
\newcommand{\thh}{\text{th}}
\newcommand{\clone}{\kappa}
\newcommand{\loose}[2]{\Lambda_{#1}(#2)}
\newcommand{\onto}{\twoheadrightarrow}
\newcommand{\defeq}{\mathbin{\vcentcolon =}}

\DeclareMathOperator{\Hom}{Hom}
\DeclareMathOperator{\F}{F}

\newcommand{\Fbr}%
   {F_{\operatorname{br}}}                 

\newcommand{\Vbr}%
   {V_{\operatorname{br}}}                 
   
\newcommand{\Pbr}%
   {P_{\operatorname{br}}}                 
	
\newcommand{\Thkern}[1]{\mathscr{K}(#1)}
   
\newcommand{\thloose}[1]%
   {\Thkern{\loose{*}{#1}}}    

\numberwithin{equation}{section}

\begin{document}

\title{On normal subgroups of the braided Thompson groups}
\date{\today}


\author{Matthew C.~B.~Zaremsky}
\address{Department of Mathematics and Statistics, University at Albany (SUNY), Albany, NY 12222}
\email{mzaremsky@albany.edu}

\begin{abstract}
 We inspect the normal subgroup structure of the braided Thompson groups $\Vbr$ and $\Fbr$. We prove that every proper normal subgroup of $\Vbr$ lies in the kernel of the natural quotient $\Vbr\onto V$, and we exhibit some families of interesting such normal subgroups. For $\Fbr$, we prove that for any normal subgroup $N$ of $\Fbr$, either $N$ is contained in the kernel of $\Fbr\onto F$, or else $N$ contains $[\Fbr,\Fbr]$. We also compute the Bieri--Neumann--Strebel invariant $\Sigma^1(\Fbr)$, which is a useful tool for understanding normal subgroups containing the commutator subgroup.
\end{abstract}

\maketitle
\thispagestyle{empty}

\section*{Introduction}

Thompson's groups $F$, $T$ and $V$ have spent the past fifty years appearing in a variety of contexts and serving as examples of groups with unique and unexpected properties. Some examples of such properties are that $T$ and $V$ are finitely presented, infinite and simple, and $F$ is torsion-free and contains free abelian subgroups of arbitrarily high rank, but is finitely presented. While $F$ is not simple it is true that $[F,F]$ is simple, and any proper quotient of $F$ is abelian. Stronger than being finitely presented, all three groups are also of type $\F_\infty$, meaning they admit classifying spaces with compact $n$-skeleta, for all $n\in\N$.

The braided Thompson groups $\Vbr$ and $\Fbr$ appeared more recently, but have proved to have many interesting properties. First, $\Vbr$ was introduced independently by Brin \cite{brin07} and Dehornoy \cite{dehornoy06}, and serves as an ``Artinification'' of $V$. In particular it is a torsion-free group with $V$ as a quotient, which contains copies of every braid group $B_n$, and is finitely presented. A subgroup $\Fbr$ of $\Vbr$ was introduced by Brady, Burillo, Cleary and Stein \cite{brady08}. This group is finitely presented, contains copies of every pure braid group $PB_n$ and has $F$ as a quotient. Both $\Vbr$ and $\Fbr$ are also of type $\F_\infty$ \cite{bux16}. The fact that these groups are so vast as to contain every braid group, while still having such nice finiteness properties, makes them of considerable interest.

In this paper we analyze the normal subgroups of $\Vbr$ and $\Fbr$. There is a natural normal subgroup $\Pbr$ of $\Vbr$, which is the kernel of the map $\Vbr \onto V$, and is also the kernel of $\Fbr\onto F$. We prove the following Alternative for $\Fbr$:

\begin{main:alt}
 Let $N$ be a normal subgroup of $\Fbr$. Then either $N\le \Pbr$ or else $[\Fbr,\Fbr]\le N$.
\end{main:alt}

This has a corollary for $\Vbr$:

\begin{main:BV_alt}
 Any proper normal subgroup of $\Vbr$ is contained in $\Pbr$.
\end{main:BV_alt}

Note that, since $V$ is simple, any $N\triangleleft \Vbr$ not contained in $\Pbr$ satisfies $N\Pbr=\Vbr$, so the corollary could also be phrased: ``Any normal subgroup of $\Vbr$ either contains or is contained in $\Pbr$.'' This was conjectured by Kai-Uwe Bux after the preprint \cite{bux08}.

A consequence of these results is that $\Vbr$ and $[\Fbr,\Fbr]$ are perfect, but not $\Fbr$, which is somewhat analogous to the classical fact that $V$ and $[F,F]$ are simple, but not $F$. Also, we obtain some pleasant statements for the braided versions that are also true (for sometimes trivial reasons) for the classical versions, like: any quotient of $\Fbr$ is either abelian or else contains $F$, and any non-trivial quotient of $\Vbr$ surjects onto $V$.

We further analyze normal subgroups of $\Fbr$ containing the commutator subgroup by computing the Bieri--Neumann--Strebel invariant $\Sigma^1(\Fbr)$. This is a geometric invariant of a finitely generated group $G$ that determines which normal subgroups containing $[G,G]$ are themselves finitely generated. In general the BNS-invariant is considered to be quite difficult to compute. We state the result here, and see Section~\ref{sec:BNS} for the notation and background.

\begin{main:BNS_BF}
 The Bieri--Neumann--Strebel invariant $\Sigma^1(\Fbr)$ for $\Fbr$ consists of all points on the sphere $\Sigma(\Fbr)=S^3$ except for the points $[\phi_0]$ and $[\phi_1]$.
\end{main:BNS_BF}

For example our calculation of $\Sigma^1(\Fbr)$ shows that for $[\Fbr,\Fbr]\le N\le \Fbr$, $N$ fails to be finitely generated if and only if it is contained in either $\ker(\phi_0)$ or $\ker(\phi_1)$, with notation explained in Section~\ref{sec:chars}.

Lastly we inspect normal subgroups of $\Vbr$ and $\Fbr$ contained in $\Pbr$. We classify how they arise, namely any such normal subgroup is the limit of a uniquely determined \emph{complete coherent} sequence of normal subgroups of the $PB_n$. Details are given in Section~\ref{sec:non_abelian_quotients}, along with some examples, and some questions. Perhaps the most tantalizing question, which we have so far been unable to answer, is whether $\Vbr$ and/or $\Fbr$ is Hopfian; $V$ and $F$ are Hopfian, but we show that $\Pbr$ is not, so it is not entirely clear what to expect.

The paper is organized as follows. In Section~\ref{sec:background} we recall the relevant background on the braided Thompson groups. In Section~\ref{sec:alternative} we prove the Alternative for $\Fbr$, Theorem~\ref{thrm:alt}. Normal subgroups of $\Fbr$ containing the commutator subgroup are further investigated in Section~\ref{sec:abelain_quotients}, where the BNS-invariant $\Sigma^1(\Fbr)$ is computed. Normal subgroups contained in $\Pbr$ are discussed in Section~\ref{sec:non_abelian_quotients}.

\subsection*{Acknowledgments} I am grateful to Robert Bieri and Matt Brin for many helpful conversations, and Marco Marschler and Stefan Witzel for their comments and suggestions. Thanks are also due to the anonymous referee for many helpful suggestions, which in particular improved Section~\ref{sec:non_abelian_quotients}.

\section{The braided Thompson groups}\label{sec:background}

In this section we recall a model for elements of $\Vbr$ and $\Fbr$, state some presentations, discuss the abelianizations of the groups (in fact $\Vbr$ is perfect), and fix some notation for characters of $\Fbr$ that will be used in Section~\ref{sec:abelain_quotients}.

\subsection{Definitions and models}\label{sec:the_groups}

Elements of $\Vbr$ are represented by \emph{braided paired tree diagrams}, as in \cite{brady08}. By a \emph{tree} we will always mean a finite rooted binary tree. The \emph{trivial tree} is just a single node. Vertices of a non-trivial tree have valency $3$, except for the \emph{leaves}, which have valency $1$, and the \emph{root}, which has valency $2$. A non-leaf vertex $u$, together with the two edges and their vertices $v,w$ connected to $u$ and directed away from the root, form a \emph{caret}. The vertices $v$ and $w$ are \emph{children} of $u$. Our trees will always come equipped with a decision for each such $u$, as to which of $v$ or $w$ is the \emph{left} or \emph{right} child. This induces a numbering of the leaves of a tree, left to right, from $1$ to $n$ for some $n$.

A braided paired tree diagram is a triple $(T_-,b,T_+)$ where $T_\pm$ are trees, each with $n$ leaves for some $n\in\N$, and $b$ is an element of the braid group $B_n$. The model we will use for elements is split-braid-merge diagrams. We draw $T_-$ (the splits) with the root on the top and the $n$ leaves at the bottom, then the braid on $n$ strands, and then $T_+$ (the merges) with the $n$ leaves at the top and the root at the bottom. See Figure~\ref{fig:element_Vbr} for an example.

\begin{figure}[htb]
 \begin{tikzpicture}[line width=0.8pt]
  \draw
   (0,-2) -- (2,0) -- (4,-2)   (0.5,-1.5) -- (1,-2)   (1,-1) -- (2,-2)   (3.5,-1.5) -- (3,-2)
   (0,-2) to [out=-90, in=90] (1,-4)   (3,-2) to [out=-90, in=90] (4,-4);
  \draw[white, line width=4pt]
   (2,-2) to [out=-90, in=90] (0,-4)   (4,-2) to [out=-90, in=90] (3,-4);
  \draw
   (2,-2) to [out=-90, in=90] (0,-4)   (1,-2) to [out=-90, in=90] (2,-4)   (4,-2) to [out=-90, in=90] (3,-4);
  \draw[white, line width=4pt]
  (1,-2) to [out=-90, in=90] (2,-4);
  \draw
  (1,-2) to [out=-90, in=90] (2,-4);
  \draw
   (0,-4) -- (2,-6) -- (4,-4)   (1.5,-5.5) -- (3,-4)   (1,-4) -- (2,-5)   (2,-4) -- (1.5,-4.5);
 \end{tikzpicture}
 \caption{An element of $\Vbr$.}\label{fig:element_Vbr}
\end{figure}
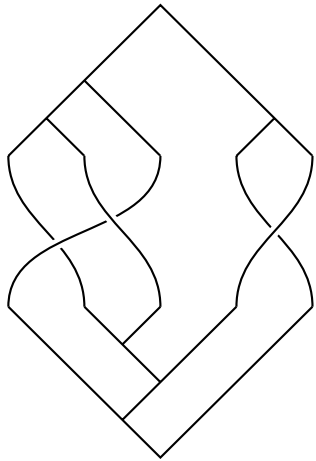

Two such triples are considered \emph{equivalent} if they are connected via a finite sequence of \emph{reductions and expansions}. An expansion of $(T_-,b,T_+)$ amounts to adding a caret to some leaf of $T_-$, bifurcating the strand coming out of that leaf into two parallel strands, and then adding a caret to the leaf of $T_+$ at which that original strand ended. A reduction is the reverse of an expansion.

We will use expansions a lot in all that follows, so we make some relevant definitions here, following \cite{witzel18}.

\begin{definition}[Cloning]\label{def:cloning}
 Let $\clone_k^n \colon B_n \to B_{n+1}$ be the injective function that takes a braid and bifurcates the $k\thh$ strand into two parallel strands, where we number the strands at the bottom. We call $\clone_k^n$ the $k\thh$ \emph{cloning map}, and we say that the resulting strands are \emph{clones}. Note that $\clone_k^n$ is not a group homomorphism, since the numbering of the strands may be different on the bottom and the top. When appropriate, we may also write $\clone_k$ for $\clone_k^n$.
 
 For a tree $T$ with $n$ leaves, let $\lambda_k$ be a single-caret tree whose root is identified with the $k\thh$ leaf of $T$. Denote by $T\cup\lambda_k$ the tree obtained by attaching this caret to that leaf. Let $\rho_b$ be the image of $b$ under the natural quotient $B_n\to \symm_n$. Now, for trees $T_-$ and $T_+$ with $n$ leaves and $b\in B_n$ we have the expansion
 $$(T_-,b,T_+)=(T_- \cup \lambda_{\rho_b(k)},\clone_k^n(b),T_+ \cup \lambda_k)\text{.}$$
 
 We can iterate this. Let $T$ be a tree with $n$ leaves and let $\Phi$ be a forest with $n$ roots. This is just a finite sequence of trees, and the \emph{roots} of the forest are the roots of its trees. We can write $\Phi$ as an ordered union of carets, $\Phi=\lambda_{k_1} \cup \cdots \cup \lambda_{k_r}$, and consider the tree $T\cup\Phi$ with $n+r$ leaves. For this to makes sense we need $1\le k_1\le n$, then $1\le k_2\le n+1$, and so forth up to $1\le k_r\le n+(r-1)$. Here by ``ordered'' union we just mean that the subscripts we need to use for the $\lambda_{k_i}$ depend on the order in which the carets are attached, e.g., $\lambda_2 \cup \lambda_1 = \lambda_1 \cup \lambda_3$ are two strings of carets both representing the forest $\Phi$ consisting of two disjoint carets. Denote by $\clone_\Phi^n$ the iterated cloning map
 $$\clone_\Phi\defeq \clone_{k_r}\circ \cdots \circ \clone_{k_1}\text{.}$$
 This is well defined, that is, if $\Phi$ can be written as a different ordered union of carets, we still get the same cloning map.
\end{definition}

In some ways it makes more sense to treat cloning maps as right maps, and write $(b)\clone_k^n$ (as in \cite{witzel18}), since otherwise as seen above we have things like $\clone_{\lambda_{k_1} \cup \lambda_{k_2}}=\clone_{\lambda_{k_2}} \circ \clone_{\lambda_{k_1}}$, but this technical precision is outweighed by future notational awkwardness, so we will stick to writing $\clone_k^n(b)$.

As an example of cloning, in Figure~\ref{fig:expansion} we have an expansion of the form
$$(T_-,b,T_+)=(T_- \cup \lambda_1, \clone_2^4(b), T_+ \cup \lambda_2)\text{.}$$

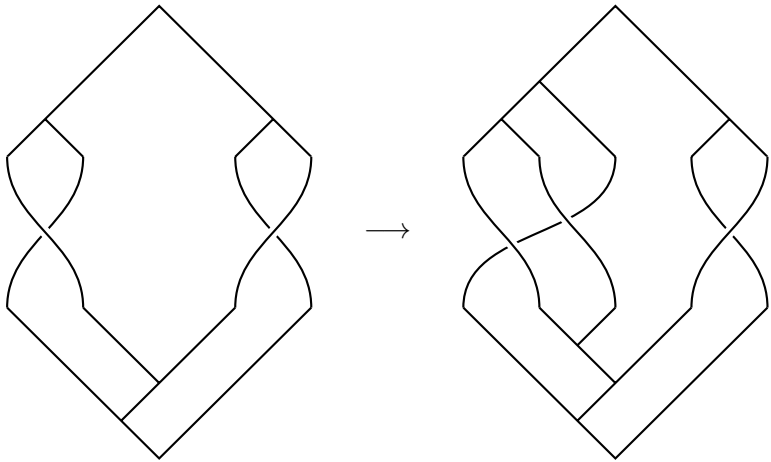
\begin{figure}[htb]
 \begin{tikzpicture}[line width=0.8pt]
  \draw
   (0,-2) -- (2,0) -- (4,-2)   (0.5,-1.5) -- (1,-2)   (3.5,-1.5) -- (3,-2)
   (1,-2) to [out=-90, in=90] (0,-4)   (3,-2) to [out=-90, in=90] (4,-4);
  \draw[white, line width=4pt]
   (0,-2) to [out=-90, in=90] (1,-4)   (4,-2) to [out=-90, in=90] (3,-4);
  \draw
   (0,-2) to [out=-90, in=90] (1,-4)   (4,-2) to [out=-90, in=90] (3,-4);
  \draw
   (0,-4) -- (2,-6) -- (4,-4)   (1.5,-5.5) -- (3,-4)   (1,-4) -- (2,-5);
   
  \node at (5,-3) {$\longrightarrow$};
  
 \begin{scope}[xshift=6cm]
  \draw
   (0,-2) -- (2,0) -- (4,-2)   (0.5,-1.5) -- (1,-2)   (1,-1) -- (2,-2)   (3.5,-1.5) -- (3,-2)
   (2,-2) to [out=-90, in=90] (0,-4)   (3,-2) to [out=-90, in=90] (4,-4);
  \draw[white, line width=4pt]
   (0,-2) to [out=-90, in=90] (1,-4)   (1,-2) to [out=-90, in=90] (2,-4)   (4,-2) to [out=-90, in=90] (3,-4);
  \draw
   (0,-2) to [out=-90, in=90] (1,-4)   (1,-2) to [out=-90, in=90] (2,-4)   (4,-2) to [out=-90, in=90] (3,-4);
  \draw
   (0,-4) -- (2,-6) -- (4,-4)   (1.5,-5.5) -- (3,-4)   (1,-4) -- (2,-5)   (2,-4) -- (1.5,-4.5);
 \end{scope}
   
 \end{tikzpicture}
 \caption{Expansion in $\Vbr$.}\label{fig:expansion}
\end{figure}

Recall that a braid $b$ is \emph{pure} if $\rho_b$ is the trivial element of $\symm_n$. We will denote the subgroup of all pure braids by $PB_n$.

\begin{observation}
 If $b\in B_n$ then $b$ is pure if and only if $\clone_k^n(b)$ is pure for all $k$. That is, the property of a braid being pure is invariant under both expansion and reduction. Restricted to $PB_n$, the cloning maps $\clone_k^n \colon PB_n \to PB_{n+1}$ are group homomorphisms, since the numbering of the strands is the same on the bottom and the top.
\end{observation}

The set of all equivalence classes of braided paired tree diagrams forms a group, $\Vbr$, with multiplication given by ``stacking'' the diagrams. By restricting to only considering pure braids, we obtain the subgroup $\Fbr$. Crucial to our model being useful is that one can always turn a product of split-braid-merge diagrams into a single split-braid merge diagram via finitely many reductions and expansions. There are some natural subgroups of $\Vbr$ and $\Fbr$ worth mentioning. First, $\Vbr$ contains a copy of $F$ (diagrams with no braiding), and ``many'' copies of every braid group $B_n$ for $n\in\N$. In particular for any tree with $n$ leaves, the set of triples $(T,b,T)$ for $b\in B_n$ is isomorphic to $B_n$. Similarly, $\Fbr$ contains $F$ and many copies of every pure braid group $PB_n$ for $n\in\N$, namely a copy of $PB_n$ for every tree with $n$ leaves.

\subsection{The kernel $\Pbr$}\label{sec:Pbr}

The group $\Vbr$ surjects onto $V$ under the map that turns every braid $b$ into the permutation $\rho_b$. The kernel of this map consists of elements represented by triples $(T,p,T)$ where $p$ is pure. Note that the two trees must \emph{both} be $T$, if $(T,p,T)$ is to become trivial under $\Vbr\onto V$. We will denote this kernel by $\Pbr$, so we have a short exact sequence
$$1\to \Pbr\to \Vbr\to V\to 1\text{.}$$

Of course $\Pbr\le \Fbr$, and is the kernel of the natural quotient $\Fbr\onto F$. The short exact sequence above restricts to
$$1\to\Pbr\to\Fbr\to F\to 1 \text{,}$$
which splits, so $\Fbr=\Pbr\rtimes F$. The sequence for $\Vbr$ does not split; for instance $V$ has torsion but $\Vbr$ is torsion-free.

The kernel $\Pbr$ is a direct limit of copies of $PB_n$, arranged in a certain directed system. This is spelled out in detail in Section~1 of \cite{burillo08}. In short, for a tree $T$ with $n$ leaves, we have an isomorphic copy of $PB_n$, denoted $PB_T$, consisting of triples $(T,p,T)$ for $p\in PB_n$. We write $T\le T'$ if $T'$ is obtained from $T$ by an iterated process of adding carets to the leaves of $T$. This makes the set of $PB_T$ into a directed system, with morphisms given by the inclusions induced by cloning maps. The limit of this system is exactly $\Pbr$. As a remark, the notation in \cite{burillo08} for $\Pbr$ is $PBV$, and the inclusions induced by the cloning maps are denoted $\alpha_{n,T,i}$.

\subsection{Presentations}\label{sec:presentations}

Brady, Burillo, Cleary and Stein \cite{brady08} give infinite and finite presentations for both $\Vbr$ and $\Fbr$. For our purposes the infinite presentations are the more useful ones.

First we look at $\Vbr$. The generators are $x_i$ ($0\le i$), $\sigma_i$ and $\tau_i$ ($1\le i$). The relations are as follows.

\

\begin{longtable}{lll}
 (A)	 & $x_jx_i = x_ix_{j+1}$ \hfill for& $0 \leq i < j$\\\\
 
 (b1) & $\sigma_i \sigma_j = \sigma_j \sigma_i$
 	\hfill for& $1\le i\le j-2$\\
 (b2) & $\sigma_i \sigma_{i+1}\sigma_i = \sigma_{i+1} \sigma_i \sigma_{i+1}$
 	\hfill for& $1\le i$\\
 (b3) & $\sigma_i \tau_j = \tau_j \sigma_i$
 	\hfill for& $1\le i\le j-2$\\
 (b4) & $\sigma_i \tau_{i+1}\sigma_i = \tau_{i+1} \sigma_i \tau_{i+1}$
 	\hfill for& $1\le i$\\\\
	
 (c1) & $\sigma_i x_j = x_j \sigma_i$
 	\hfill for& $1\le i<j$\\
 (c2) & $\sigma_i x_i = x_{i-1} \sigma_{i+1} \sigma_i$
 	\hfill for& $1\le i$\\
 (c3) & $\sigma_i x_j = x_j \sigma_{i+1}$
 	\hfill for& $1\le j\le i-2$\\
 (c4) & $\sigma_{i+1} x_i = x_{i+1} \sigma_{i+1} \sigma_{i+2}$
 	\hfill for& $1\le i$\\\\
	
 (d1) & $\tau_i x_j = x_j \tau_{i+1}$
 	\hfill for& $1\le j\le i-2$\\
 (d2) & $\tau_i x_{i-1} = \sigma_i \tau_{i+1}$
 	\hfill for& $1\le i$\\
 (d3) & $\tau_i  = x_{i-1} \tau_{i+1} \sigma_i$
 	\hfill for& $1\le i$
\end{longtable}

\begin{figure}[htb]
 \begin{tikzpicture}[line width=0.8pt]
  \draw
   (0,-2) -- (2,0) -- (4,-2)   (2.5,-0.5) -- (1,-2)   (3,-1) -- (2,-2)   (3.5,-1.5) -- (3,-2);
   
  \draw (0,-2) -- (0,-4)   (1,-2) -- (1,-4)   (4,-2) -- (4,-4)
   (3,-2) to [out=-90, in=90] (2,-4);
  \draw[white, line width=4pt]
   (2,-2) to [out=-90, in=90] (3,-4);
  \draw
   (2,-2) to [out=-90, in=90] (3,-4);
  
  \draw
   (0,-4) -- (2,-6) -- (4,-4)   (2.5,-5.5) -- (1,-4)   (3,-5) -- (2,-4)   (3.5,-4.5) -- (3,-4);
  \node at (-0.6,-3) {$\sigma_3$:};
  
  \begin{scope}[xshift=6cm]
   \draw
   (0,-2) -- (2,0) -- (4,-2)   (2.5,-0.5) -- (1,-2)   (3,-1) -- (2,-2)   (3.5,-1.5) -- (3,-2);
   
  \draw (0,-2) -- (0,-4)   (1,-2) -- (1,-4)   (2,-2) -- (2,-4)
   (4,-2) to [out=-90, in=90] (3,-4);
  \draw[white, line width=4pt]
   (3,-2) to [out=-90, in=90] (4,-4);
  \draw
   (3,-2) to [out=-90, in=90] (4,-4);
  
  \draw
   (0,-4) -- (2,-6) -- (4,-4)   (2.5,-5.5) -- (1,-4)   (3,-5) -- (2,-4)   (3.5,-4.5) -- (3,-4);
  \node at (-0.6,-3) {$\tau_4$:};
  \end{scope}

 \end{tikzpicture}
 \caption{Examples of generators of $\Vbr$.}\label{fig:BV_gens}
\end{figure}

The elements $x_i$ are the standard generators of $F$. The $\sigma_i$ are given by $(R_{i+2},a_i,R_{i+2})$, where $R_{i+2}$ is the \emph{all-right tree} with $i+2$ leaves, and $a_i\in B_{i+2}$ is the braid that crosses strand $i$ across strand $i+1$. An all-right tree is one in which, for every caret but the first, that caret's root is the previous caret's right leaf. The $\tau_i$ are given by $(R_{i+1},b_i,R_{i+1})$, where $b_i\in B_{i+1}$ crosses strand $i$ across strand $i+1$. The important difference is that in $\sigma_i$ the last strand is not used, and in $\tau_i$ it is. See Figure~\ref{fig:BV_gens} for some examples.

Now we look at $\Fbr$. The generators are $x_i$ ($0\le i$), $\alpha_{i,j}$ and $\beta_{i,j}$ ($1\le i<j$). The relations are as follows.

\

\begin{longtable}{lll}
 (A)	 & $x_jx_i = x_ix_{j+1}$ \hfill for& $0 \leq i < j$\\\\

 (B1)	 & $\alpha_{r,s}^{-1}\alpha_{i,j}\alpha_{r,s} = \alpha_{i,j}$
	 \hfill for& $1 \leq r < s < i < j$ \\  &\hfill or& $1 \leq i < r < s < j$\\
 (B2)	 & $\alpha_{r,s}^{-1}\alpha_{i,j}\alpha_{r,s} = \alpha_{r,j}\alpha_{i,j}\alpha_{r,j}^{-1}$ 
	 \hfill for& $1 \leq r < s = i < j$\\
 (B3)	 & $\alpha_{r,s}^{-1}\alpha_{i,j}\alpha_{r,s} = (\alpha_{i,j}\alpha_{s,j})\alpha_{i,j}(\alpha_{i,j}\alpha_{s,j})^{-1}$ 
	 \hfill for& $1 \leq r = i < s < j$\\
 (B4)	 & $\alpha_{r,s}^{-1}\alpha_{i,j}\alpha_{r,s} 	= (\alpha_{r,j}\alpha_{s,j}\alpha_{r,j}^{-1}\alpha_{s,j}^{-1})    \alpha_{i,j} (\alpha_{r,j}\alpha_{s,j}\alpha_{r,j}^{-1}\alpha_{s,j}^{-1})^{-1}$ 
	 \hfill for& $1 \leq r < i < s < j$\\
 (B5)	 & $\alpha_{r,s}^{-1}\beta_{i,j}\alpha_{r,s} = \beta_{i,j}$
	 \hfill for& $1 \leq r < s < i < j$ \\&\hfill or& $1 \leq i < r < s < j$\\
 (B6)	 & $\alpha_{r,s}^{-1}\beta_{i,j}\alpha_{r,s} = \beta_{r,j}\beta_{i,j}\beta_{r,j}^{-1}$
	 \hfill for& $1 \leq r < s = i < j$\\
 (B7)	 & $\alpha_{r,s}^{-1}\beta_{i,j}\alpha_{r,s} = (\beta_{i,j}\beta_{s,j})\beta_{i,j}(\beta_{i,j}\beta_{s,j})^{-1}$
	 \hfill for& $1 \leq r = i < s < j$\\
 (B8)	 & $\alpha_{r,s}^{-1}\beta_{i,j}\alpha_{r,s} = (\beta_{r,j}\beta_{s,j}\beta_{r,j}^{-1}\beta_{s,j}^{-1})    \beta_{i,j} (\beta_{r,j}\beta_{s,j}\beta_{r,j}^{-1}\beta_{s,j}^{-1})^{-1}$
	 \hfill for& $1 \leq r < i < s < j$\\\\

 (C)	& $\beta_{i,j} = \beta_{i,j+1}\alpha_{i,j}$ 
	\hfill for& $1 \leq i < j$\\\\

 (D1) & $\alpha_{i,j}x_{k-1} = x_{k-1}\alpha_{i+1,j+1}$
 	 \hfill for& $1 \leq k < i < j$\\
 (D2) & $\alpha_{i,j}x_{k-1} = x_{k-1}\alpha_{i+1,j+1}\alpha_{i,j+1}$
 	 \hfill for& $1 \leq k = i < j$\\
 (D3) & $\alpha_{i,j}x_{k-1} = x_{k-1}\alpha_{i,j+1}$
 	 \hfill for& $1 \leq  i < k < j$\\
 (D4) & $\alpha_{i,j}x_{k-1} = x_{k-1}\alpha_{i,j+1}\alpha_{i,j}$
 	 \hfill for& $1 \leq i < k = j$\\
 (D5) & $\alpha_{i,j}x_{k-1} = x_{k-1}\alpha_{i,j}$
 	 \hfill for& $1 \leq  i < j < k$\\
 (D6) & $\beta_{i,j}x_{k-1} = x_{k-1}\beta_{i+1,j+1}$
 	 \hfill for& $1 \leq k < i < j$\\
 (D7) & $\beta_{i,j}x_{k-1} = x_{k-1}\beta_{i+1,j+1}\beta_{i,j+1}$
 	 \hfill for& $1 \leq k = i < j$\\
 (D8) & $\beta_{i,j}x_{k-1} = x_{k-1}\beta_{i,j+1}$
 	 \hfill for& $1 \leq  i < k < j$\\
 (D9) & $\beta_{i,j}x_{k-1} = x_{k-1}\beta_{i,j}$
 	 \hfill for& $1 \leq  i < j < k$.
\end{longtable}

\

These generators can be written in terms of the generators for $\Vbr$ as follows:

\begin{align*}
 \alpha_{i,j}&=\sigma_i \sigma_{i+1} \cdots \sigma_{j-2} \sigma_{j-1}^2 \sigma_{j-2}\I \cdots \sigma_i\I \text{ and }\\
\beta_{i,j}&=\sigma_i \sigma_{i+1} \cdots \sigma_{j-2} \tau_{j-1}^2 \sigma_{j-2}\I \cdots \sigma_i\I \text{.}
\end{align*}

Pictorially, $\alpha_{i,j}$ is an all-right tree of splits, out to $j+1$ strands, then strand $i$ braids around strand $j$ and goes back to position $i$, and then there is an all-right tree of merges. For $\beta_{i,j}$ the only difference is that we go out to $j$ strands, and so we use the last strand. See Figure~\ref{fig:BF_gens} for some examples.

\begin{figure}[hbt]
 \begin{tikzpicture}[line width=0.8pt]
  \draw
   (0,-2) -- (2,0) -- (4,-2)   (2.5,-0.5) -- (1,-2)   (3,-1) -- (2,-2)   (3.5,-1.5) -- (3,-2);
   
  \draw (0,-2) -- (0,-4)   (2,-2) -- (2,-4)   (4,-2) -- (4,-4);
  \draw[white, line width=4pt]
   (3.5,-3) to [out=-90, in=90] (1,-4);
  \draw
   (3.5,-3) to [out=-90, in=90] (1,-4);
  \draw[white, line width=4pt]
   (3,-2) -- (3,-4);
  \draw
   (3,-2) -- (3,-4);
  \draw[white, line width=4pt]
   (1,-2) to [out=-90, in=90] (3.5,-3);
  \draw
   (1,-2) to [out=-90, in=90] (3.5,-3);
  
  \draw
   (0,-4) -- (2,-6) -- (4,-4)   (2.5,-5.5) -- (1,-4)   (3,-5) -- (2,-4)   (3.5,-4.5) -- (3,-4);
  \node at (-0.6,-3) {$\alpha_{2,4}$:};
  
  \begin{scope}[xshift=6cm]
   \draw
   (0,-2) -- (2,0) -- (4,-2)   (2.5,-0.5) -- (1,-2)   (3,-1) -- (2,-2)   (3.5,-1.5) -- (3,-2);
   
  \draw (0,-2) -- (0,-4)   (1,-2) -- (1,-4)   (3,-2) -- (3,-4);
  \draw[white, line width=4pt]
   (4.5,-3) to [out=-90, in=90] (2,-4);
  \draw
   (4.5,-3) to [out=-90, in=90] (2,-4);
  \draw[white, line width=4pt]
   (4,-2) -- (4,-4);
  \draw
   (4,-2) -- (4,-4);
  \draw[white, line width=4pt]
   (2,-2) to [out=-90, in=90] (4.5,-3);
  \draw
   (2,-2) to [out=-90, in=90] (4.5,-3);
  
  \draw
   (0,-4) -- (2,-6) -- (4,-4)   (2.5,-5.5) -- (1,-4)   (3,-5) -- (2,-4)   (3.5,-4.5) -- (3,-4);
  \node at (-0.6,-3) {$\beta_{3,5}$:};
  \end{scope}

 \end{tikzpicture}
 \caption{Examples of generators of $\Fbr$.}\label{fig:BF_gens}
\end{figure}

\begin{observation}\label{obs:Pbr_gens}
 The group $\Pbr$ is generated by the conjugates of $\alpha_{i,j}$ and $\beta_{i,j}$ ($1\le i<j$) by elements of $F$.
\end{observation}

\begin{proof}
 We need to generate every $PB_T$. If $T$ has $n$ leaves and $R_n$ is the all-right tree with $n$ leaves, then $PB_T$ is conjugate to $PB_{R_n}$ by the element $(R_n,T)$ of $F$. But $PB_{R_n}$ is generated by the $\alpha_{i,j}$ for $1\le i<j<n$ and $\beta_{i,j}$ for $1\le i<j\le n$.
\end{proof}

It will be convenient in Section~\ref{sec:abelain_quotients} to use a fact about $\Fbr$ that does not seem to have been recorded before, namely that it is an ascending HNN-extension of a certain subgroup. For $n\ge0$ let $F(n)$ be the subgroup of $F$ generated by all the $x_i$ with $i\ge n$. In particular $F(0)=F$. It is well known that $F$ is an ascending HNN-extension of $F(1)$ with stable element $x_0$. Now define $\Fbr(n)$ to be the subgroup of $\Fbr$ generated by all the $\alpha_{i,j}$ and $\beta_{i,j}$ for $1\le i<j$ and all the $x_i$ for $i\ge n$.

\begin{lemma}\label{lem:hnn}
 We have that $\Fbr$ is an ascending HNN-extension of $\Fbr(1)$ with stable element $x_0$.
\end{lemma}

\begin{proof}
 An initial proof involved establishing a presentation for $\Fbr(1)$, and was more involved; the following faster proof is inspired by helpful discussions with Robert Bieri and Matt Brin.

 Let $\theta \colon \Fbr(1) \hookrightarrow \Fbr(1)$ be the monomorphism given by right conjugation by $x_0$. There is an epimorphism $\Psi$ from the abstract HNN-extension $\Fbr(1) \ast_\theta$, with stable element $t$, to $\Fbr$ given by specializing $t$ to $x_0$, so $\Psi$ is the identity map on $\Fbr(1)$. We need to check that $\Psi$ is injective. Let $g\in\ker(\Psi)$. Since $g$ is an element of $\Fbr(1) \ast_\theta$, and since the HNN-extension is ascending, we can write $g$ in the form $g=t^n h t^m$ for $h\in\Fbr(1)$, $n\ge0$ and $m\le0$. Also, there is an epimorphism $\Fbr(1) \ast_\theta \onto \Z$ that reads $0$ on $\Fbr(1)$ and $1$ on $t$, and factors through $\Psi$; this tells us that $m=-n$. Hence $h$ is itself in $\ker(\Psi)$. But $\Psi\mid_{\Fbr(1)}$ is the identity, so $h=1$ and we are done.
\end{proof}

\subsection{Abelianization and characters}\label{sec:chars}

The first observation of this subsection is about $\Vbr$, and ensures that any future discussion about abelian quotients and characters will be uninteresting for $\Vbr$.

\begin{observation}\label{obs:BV_perfect}
 The group $\Vbr$ is perfect.
\end{observation}

This is an easy exercise in abelianizing the presentation for $\Vbr$ from the previous subsection and checking that every generator becomes trivial.

One can also abelianize the presentation for $\Fbr$ without too much difficulty. For reference we will describe the steps in the following lemma. Here, bars indicate the images of elements in the abelianization.

\begin{lemma}
 The abelianization of $\Fbr$ is generated by $\overline{x}_0$, $\overline{x}_1$, $\overline{\beta}_{1,3}$ and $\overline{\alpha}_{1,2}$.
\end{lemma}

\begin{proof}
 We start with generators $\overline{x}_i$ ($0\le i$), $\overline{\alpha}_{i,j}$ ($1\le i<j$) and $\overline{\beta}_{i,j}$ ($1\le i<j$). Relation (A) tells us that $\overline{x}_i=\overline{x}_1$ for all $i\ge2$. Relations (D1), (D3), (D6) and (D8) tell us that each $\overline{\alpha}_{k,\ell}$ equals $\overline{\alpha}_{i,j}$ for some
 $$(i,j) \in \{(1,2),(1,3),(2,3),(2,4)\} \text{,}$$
 with a similar statement for the $\overline{\beta}_{k,\ell}$. So far we have reduced down to ten generators. We will show that six of them are redundant, leaving the four in the statement of the lemma.
 
 From (D4) we see that $\overline{\alpha}_{1,3}=0$ and $\overline{\alpha}_{2,4}=0$. Then (D2) says that $\overline{\alpha}_{2,3} = \overline{\alpha}_{1,2}$. From (C) and (D7) we see that $\overline{\beta}_{i+1,j+1}=\overline{\alpha}_{i,j}$ for all $1\le i<j$, which implies that $\overline{\beta}_{2,4}=0$ and $\overline{\beta}_{2,3}=\overline{\alpha}_{1,2}$. Finally, (C) says that $\overline{\beta}_{1,2}=\overline{\beta}_{1,2}+\overline{\alpha}_{1,2}$.
\end{proof}

In order to prove that these four generators are in fact linearly independent, and so $\Fbr$ abelianizes to $\Z^4$,  we first describe four discrete characters of $\Fbr$, denoted $\phi_0$, $\phi_1$, $\omega_0$ and $\omega_1$, which will be dual to certain combinations of these generators. Recall that a \emph{character} of a group is a homomorphism from the group to the additive real numbers, and a character is \emph{discrete} if its image is isomorphic to $\Z$.

First, note that $\Fbr$ acts on $[0,1]$ via the map $\pi\colon\Fbr\onto F$. A standard basis for $\Hom(F,\R)\cong \R^2$ is $\{\chi_0, \chi_1\}$, where $\chi_i(f)\defeq \log_2(f'(i))$ \cite{bieri10}. Hence we get two linearly independent characters for $\Fbr$ by composing, namely:
$$\phi_0\defeq \chi_0 \circ \pi \text{ and } \phi_1\defeq \chi_1 \circ \pi\text{.}$$

The values of $\phi_0$ and $\phi_1$ can be read off a representative triple $(T_-,p,T_+)$ for an element of $\Fbr$. For a tree $T$, thought of as a metric graph with edge lengths all $1$, let $r$ be the root, $\ell_\ell$ the leftmost leaf and $\ell_r$ the rightmost leaf. Let $L(T)$ be the length of the reduced edge path from $r$ to $\ell_\ell$, and $R(T)$ the length of the reduced edge path from $r$ to $\ell_r$. Then
$$\phi_0(T_-,p,T_+) = L(T_+)-L(T_-) \text{ and } \phi_1(T_-,p,T_+) = R(T_+)-R(T_-)\text{.}$$

The characters $\phi_0$ and $\phi_1$ both have $\Pbr$ in their kernels. To find the missing two dimensions in what we will eventually see is $\Hom(\Fbr,\R)\cong \R^4$, we now look at characters that can detect braiding. First, let $\omega_0$ be the character that takes an element $(T_-,p,T_+)$ and reads off the total winding number of the first and last strands of $p$ around each other. This is invariant under reduction and expansion, and so is well defined. Finally, let $\omega_1$ be the character that reads off the sum of the total winding numbers of adjacent strands of $p$, i.e., $1$ and around $2$, plus $2$ around $3$, etc. This is again invariant under reduction and expansion, so is well defined.

As an example of these measurements, one can compute that $\phi_0(g)=1$, $\phi_1(g)=0$, $\omega_0(g)=1$ and $\omega_1(g)=-1$ for the element $g$ pictured in Figure~\ref{fig:BF_measurements}.

\begin{figure}[htb]
 \centering
  \begin{tikzpicture}[line width = 1pt]
   \coordinate (A) at (0,0);
   \coordinate (B) at (-0.333,-0.333);
   \coordinate (C) at (-1,-1);
   \coordinate (D) at (0,-0.667);
   \coordinate (E) at (-0.333,-1);
   \coordinate (F) at (0.333,-1);
   \coordinate (G) at (1,-1);
   
   \coordinate (H) at (0,-5);
   \coordinate (I) at (-0.333,-4.667);
   \coordinate (J) at (-0.667,-4.333);
   \coordinate (K) at (-1,-4);
   \coordinate (L) at (-0.333,-4);
   \coordinate (M) at (0.333,-4);
   \coordinate (N) at (1,-4);
   \coordinate (el) at (-0.333,-2.5);
   \coordinate (ck) at (1.333,-2.5);
   
   \draw (C) -- (A) -- (G)   (B) -- (F)   (E) -- (D);
   \draw (K) -- (H) -- (N)   (M) -- (I)   (L) -- (J);
   
   \draw (F) -- (M)   (el) -- (L);
   \draw[white, line width=3pt] (K) to [out=90, in=-90] (ck);
   \draw (K) to [out=90, in=-90] (ck);
   \draw[white, line width=3pt] (G) -- (N);
   \draw (G) -- (N);
   \draw[white, line width=3pt] (C) to [out=-90, in=90] (ck);
   \draw (C) to [out=-90, in=90] (ck);
   \draw[white, line width=3pt] (E) -- (el);
   \draw (E) -- (el);
  \end{tikzpicture}
 \caption{For the pictured element $g$, we have $\phi_0(g)=1$, $\phi_1(g)=0$, $\omega_0(g)=1$ and $\omega_1(g)=-1$.}
 \label{fig:BF_measurements}
\end{figure}
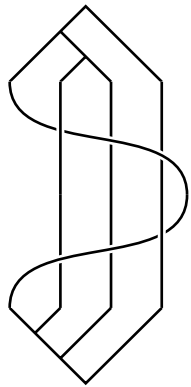

\begin{lemma}\label{lem:Hom(BF,R)}
 The characters $(\phi_0, \phi_1, \omega_0, \omega_1)$ form a basis for $\Hom(\Fbr,\R)\cong\R^4$, the elements $(\overline{x}_1 - \overline{x}_0, -\overline{x}_1, \overline{\beta}_{1,3}, \overline{\alpha}_{1,2})$ form a basis for the abelianization $\Z^4$ of $\Fbr$, and these bases are dual.
\end{lemma}

\begin{proof}
 Set $f_1\defeq \phi_0$, $f_2\defeq \phi_1$, $f_3\defeq \omega_0$ and $f_4\defeq \omega_1$. Also set $e_1\defeq \overline{x}_1 - \overline{x}_0$, $e_2\defeq -\overline{x}_1$, $e_3\defeq \overline{\beta}_{1,3}$ and $e_4 \defeq \overline{\alpha}_{1,2}$. Then for $1\le i,j\le 4$ one can check that $f_i(e_j)=\delta_{ij}$, the Kronecker delta. Here we have extended the definitions of the characters to accepting inputs from the abelianization, which is fine since they vanish on the commutator subgroup. This proves that the $e_i$ are linearly independent, so form a basis. From this it follows that $\Hom(\Fbr,\R)\cong \R^4$, and since the $f_i$ are linearly independent they form a basis.
\end{proof}

As a remark, most winding number measurements are \emph{not} invariant under reduction and expansion, for instance the total winding number of the first and second strands; on $\beta_{1,2}$ this measurement reads $1$, but when we bifurcate the first strand, we get a braid in which the first and second strands do not wind (indeed are parallel), so the measurement reads $0$. Of course now the second and third strands wind, and $\omega_1$ still reads $1$, as it is a sum over all pairs of adjacent strands. The important point is that the designations ``second,'' ``third,'' etc. are not well behaved under cloning, but ``next,'' ``first'' and ``last'' are.

We should point out that we now have an easy algorithm to check whether an element of $\Fbr$ is in $[\Fbr,\Fbr]$, namely if and only if it lies in the kernels of $\phi_0$, $\phi_1$, $\omega_0$ and $\omega_1$.

\section{An Alternative result}\label{sec:alternative}

This section is almost entirely about $\Fbr$, with implications for $\Vbr$ relegated to the end. The upshot for $\Vbr$ is Corollary~\ref{cor:BV_alt}, which says that every proper normal subgroup is contained in $\Pbr$.

The main result of this section is the following:

\begin{theorem}[$\Fbr$ Alternative]\label{thrm:alt}
 Let $N$ be a normal subgroup of $\Fbr$. Then either $N\le \Pbr$ or else $[\Fbr,\Fbr]\le N$.
\end{theorem}

By an ``Alternative result'' we mean a statement that any subgroup (or here any normal subgroup) must have one of two ``quite different'' forms. For example, the classical Tits Alternative for a group says that any subgroup either contains a non-abelian free group, or else is virtually solvable (note that Thompson's groups do not satisfy the Tits Alternative).

A quick corollary to Theorem~\ref{thrm:alt} is the following:

\begin{corollary}\label{cor:BF'_perfect}
 The commutator subgroup $[\Fbr,\Fbr]$ is perfect.
\end{corollary}

\begin{proof}
 We just need to show that $[[\Fbr,\Fbr],[\Fbr,\Fbr]]$ is not contained in $\Pbr$. But $[[\Fbr,\Fbr],[\Fbr,\Fbr]]$ contains $[[F,F],[F,F]]$, so this is clear.
\end{proof}

Note that $[\Fbr,\Fbr]$ is not simple, as it has $[\Pbr,\Pbr]$ as a proper non-trivial normal subgroup.

There are various Alternatives known for $F$ of the form, ``every subgroup of $F$ either has property $\mathcal{P}$, or else contains a copy of $G$,'' for some property $\mathcal{P}$ and some subgroup $G$. Examples of this phenomenon include:

\begin{enumerate}
 \item For $\mathcal{P}$ the property of being abelian, $G$ is $\Z^\infty$ \cite{brin85}.
 
 \item For $\mathcal{P}$ the property of being solvable, $G$ is Bleak's group $W$ \cite{bleak09}.
 
 \item The Brin--Sapir Conjecture is that this phenomenon occurs for $\mathcal{P}$ being the property of being elementary amenable and $G$ being $F$ \cite[Conjecture~3]{brin05}.
\end{enumerate}

In general, understanding the subgroups of $F$ is an active and ongoing endeavor.

Returning to the task at hand, to prove Theorem~\ref{thrm:alt}, we begin with a technical proposition about the normal closure of elements of $F$ in $\Fbr$.

\begin{proposition}\label{prop:normal_closure_F}
 Let $1\neq f\in F\le \Fbr$. The normal closure $N\defeq \langle\!\langle f\rangle\!\rangle$ of $f$ in $\Fbr$ contains $[\Fbr,\Fbr]$.
\end{proposition}

The first part of our proof is inspired by the proof of Lemma~20 in \cite{bux08}, which says that the normal closure of $[F,F]$ in $\Vbr$ is all of $\Vbr$.

\begin{proof}
 First note that since $F$ has trivial center, without loss of generality $f\in[F,F]$, and since $[F,F]$ is simple, $[F,F]\le N$. Elements of the form $x_i x_j\I$ and $x_i\I x_j$ for $i,j\ge 1$ are in $[F,F]$, and hence in $N$. This tells us that for any $i\ge 1$, the element $\alpha_{i,i+1}x_i x_{i+2}\I \alpha_{i,i+1}\I$ is in $N$. Applying (D4) and (D5), we get
 $$N \ni x_i \alpha_{i,i+2} \alpha_{i,i+1} x_{i+2}\I \alpha_{i,i+1}\I = x_i \alpha_{i,i+2} x_{i+2}\I$$
 and so $\alpha_{i,i+2} x_{i+2}\I x_i \in N$, whence $\alpha_{i,i+2} \in N$. This holds for all $i\ge 1$, and thanks to (D3), we also get that $\alpha_{i,j}\in N$ for all $1\le i<j-1$.
 
 The next goal is to force enough $\beta_{i,j}$ to be in $N$. Running a similar trick as above, we start with $\beta_{i,i+2} x_{i-1} x_{i+2}\I \beta_{i,i+2}\I$ being in $N$, use (D7) and (D9), and get that $N$ contains $\beta_{i+1,i+3} \beta_{i,i+3} \beta_{i,i+2}\I$ for all $i\ge 2$. Using (C) we get that $N$ contains $\beta_{i+1,i+3} \beta_{i,i+2} \alpha_{i,i+2}\I \beta_{i,i+2}\I$, and then since $\alpha_{i,i+2}$ is in $N$, so is $\beta_{i+1,i+3}$. Using (D6) and (D8) then, we see that $N$ contains every $\beta_{i,j}$ for $2\le i<j-1$.
 
 It now suffices to prove that upon modding out $[F,F]$ and all the $\alpha_{i,j}$ for $1\le i<j-1$ and $\beta_{i,j}$ for $2\le i<j-1$, the presentation becomes abelian. Denote elements of this quotient by putting hats on the elements (we have reserved bars for the abelianization). That all the $\widehat{x}_i$ commute follows since we have modded out $[F,F]$. Note that $\widehat{\alpha}_{r,r+2}=\widehat{1}$ for all $r\ge1$, so by (B1) and (B2) we see that all the $\widehat{\alpha}_{i,i+1}$ commute. Also, since $\widehat{\beta}_{r,r+2}=\widehat{1}$ for all $r\ge2$, by (C) we have $\widehat{\beta}_{i,i+1}=\widehat{\alpha}_{i,i+1}$ for $i\ge2$. Next we claim that every $\widehat{x}_{k-1}$ commutes with every $\widehat{\alpha}_{i,i+1}$ (for $k,i\ge1$). If $k\ge i+1$ this follows from (D4) or (D5). If $k\le i$ then (D1) and (D2) say
 $$\widehat{\alpha}_{i,i+1} \widehat{x}_{k-1} = \widehat{x}_{k-1} \widehat{\alpha}_{i+1,i+2} \text{.}$$
 Multiplying on the right by $\widehat{x}_{i+2}\I$, and using (D5) and the fact that $\widehat{x}_{k-1} \widehat{x}_{i+2}\I = \widehat{1}$, this becomes
 $$\widehat{\alpha}_{i,i+1} = \widehat{\alpha}_{i+1,i+2}\text{,}$$
 so the claim is proved.
 
 The last thing to show is that $\widehat{\beta}_{1,j}$ commutes with all the other generators. If $j>2$ then (D8) and (D9), and conjugation by $\widehat{x}_1 \widehat{x}_{j+1}^{-1}=\widehat{1}$, tell us that $\widehat{\beta}_{1,j}=\widehat{\beta}_{1,j+1}$. So we only need to look at $\widehat{\beta}_{1,2}$ and $\widehat{\beta}_{1,3}$. First note that by (C), $\widehat{\beta}_{1,2}=\widehat{\beta}_{1,3}\widehat{\alpha}_{1,2}$, so if $\widehat{\beta}_{1,3}$ commutes with everything (including $\widehat{\alpha}_{1,2}$), then so will $\widehat{\beta}_{1,2}$. Now, $\widehat{\beta}_{1,3}$ commutes with every $\widehat{x}_{k-1}$, by (D7)--(D9), and using that identification of $\widehat{\beta}_{1,3}$ with every $\widehat{\beta}_{1,j}$ for $j\ge3$. We also need an \emph{ad hoc} argument that $\widehat{\beta}_{1,3}$ commutes with $\widehat{x}_2$, which is easily checked (and holds even without the hats). Using (B5) we get that $\widehat{\beta}_{1,3}$ commutes with every $\widehat{\alpha}_{i,i+1}=\widehat{\beta}_{i,i+1}$ for $i\ge2$. Lastly, this fact plus (B7) tells us that $\widehat{\beta}_{1,3}$ commutes with $\widehat{\alpha}_{1,2}$.
\end{proof}

Thanks to this proposition, we see that ``catching'' a non-trivial element of $F$ is a way to blow up a normal subgroup to contain $[\Fbr,\Fbr]$. To prove the Alternative then, the goal is to start with an element of $\Fbr\setminus\Pbr$ and ``catch'' a non-trivial element of $F$ in its normal closure. First we need some technical lemmas.

Let $g=(T,p,T)\in\Pbr$. The tree $T$ defines a partition of $[0,1]$ into dyadic subintervals; let $X(T)$ be the set of endpoints of said subintervals. Each subinterval corresponds to a leaf of $T$, and hence to a strand of $p$. If the endpoint $x\in X(T)$ lies in $(0,1)$ and is such that the two subintervals on either side of $x$ correspond to strands of $p$ that are clones, then call $x$ \emph{inessential}. Otherwise call $x$ \emph{essential}. Let $X_{ess}(T,p)$ be the set of essential endpoints of the subintervals determined by $T$. The next observation justifies denoting this set by $X_{ess}(g)$.

\begin{observation}\label{obs:essential_invariant}
 The set $X_{ess}(g)$ defined above is an invariant of $g$.
\end{observation}

\begin{proof}
 We need to show that $X_{ess}(T,p)$ is invariant under reduction and expansion. Let $T'=T\cup\lambda_k$. We have $g=(T',\clone_k(p),T')$. On the level of subintervals, all we have done is cut the $k\thh$ subinterval in half, and so $|X(T',\clone_k(p))|=|X(T,p)|+1$. Let $x'$ be the element of $X(T',\clone_k(p))\setminus X(T,p)$, i.e., the midpoint of the $k\thh$ subinterval. Since the strands in $\clone_k(p)$ on either side of $x'$ are clones, we know that $x'$ is inessential. This shows that $X_{ess}(T,p)=X_{ess}(T',\clone_k(p))$, which tells us that the set is invariant under reduction and expansion.
\end{proof}

Note that $X_{ess}(g)=\{0,1\}$ if and only if $g=1$ in $\Pbr$.

\begin{proposition}[Commuting condition]\label{prop:commuting}
 Let $g\in\Pbr$ and $f\in F$. If $f$ fixes $X_{ess}(g)$ then $[g,f]=1$.
\end{proposition}

\begin{proof}
 Choose a tree $T$, say with $n$ leaves, such that $g=(T,p,T)$ for some $p\in PB_n$. The tree gives us a subdivision of $[0,1]$, say with endpoints $0=x_0<x_1<\cdots<x_{n-1}<x_n=1$. For $1\le i\le n$ let $I_i\defeq [x_{i-1},x_i]$. Let $0=x_{i_0}<x_{i_1}<\cdots<x_{i_{r-1}}<x_{i_r}=1$ be precisely the essential endpoints, so $|X_{ess}(g)|=r-1$. For $1\le s\le r$ define
 $$J_s\defeq \bigcup\limits_{i_{s-1}\le i\le i_s}I_i \text{,}$$
 so the $J_s$ are the closures of the connected components of $[0,1]\setminus X_{ess}(g)$. The $J_s$ partition the set of intervals $I_i$, and hence partition the leaves of $T$. For a given $J_s$, the strands of $p$ indexed by the subintervals contained in $J_s$ are all clones of each other.
 
 Now, the fact that $f$ fixes $X_{ess}(g)$ means that it can be represented by a tree pair of the form $(T \cup \Phi,T \cup \Phi')$, where $\Phi$ is a forest whose roots are identified with the leaves of $T$, as is $\Phi'$, such that a certain important property holds. To state the property we need some setup. Write $\Phi$ as $\Phi_1 \cup \cdots \cup \Phi_r$, where $\Phi_s$ is the subforest whose roots are precisely those roots of $\Phi$ identified with the leaves of $T$ lying in $J_s$. One might call $\Phi_s$ the subforest of $\Phi$ with ``support'' in $J_s$. Similarly define $\Phi'_s$ for $1\le s \le r$. Now, the important property of $\Phi$ and $\Phi'$, which we get since $f$ fixes $X_{ess}(g)$, is that for each $1\le s\le r$, the leaves of $\Phi_s$ are in bijection with the leaves of $\Phi'_s$. This bijection preserves the order on the leaves, and is induced by $f$. The ``paired forest diagram'' $(\Phi_s,\Phi'_s)$ describes how $f$ acts on the interval $J_s$.
 
 A consequence of all the above is that $\clone_{\Phi}(p)=\clone_{\Phi'}(p)$ (see Definition~\ref{def:cloning}). Indeed, for any $J_s$ the strands of $p$ indexed in $J_s$ are clones of each other, and applying $\clone_\Phi$ further clones this block of strands into a number of strands equal to the number of leaves of $\Phi_s$. This is true of $\clone_{\Phi'}$ as well, since $\Phi_s$ and $\Phi'_s$ have the same number of leaves. Then since this holds for all $s$, we conclude that $\clone_{\Phi}(p)=\clone_{\Phi'}(p)$.
 
 The following calculation finishes the proof.
 
 \begin{align*}
  fgf^{-1} &= (T \cup \Phi,T \cup \Phi')(T,p,T)(T \cup \Phi',T \cup \Phi) \\
  &= (T \cup \Phi,T \cup \Phi')(T \cup \Phi',\clone_{\Phi'}(p),T \cup \Phi')(T \cup \Phi',T \cup \Phi) \\
  &= (T \cup \Phi,\clone_{\Phi'}(p),T \cup \Phi) \\
  &= (T \cup \Phi,\clone_{\Phi}(p),T \cup \Phi) \\
  &= (T,p,T) = g \text{.}
 \end{align*}
 
\end{proof}

Figure~\ref{fig:commuting} gives an indication of what is really happening in Proposition~\ref{prop:commuting}.

\begin{figure}[htb]
 \centering
  \begin{tikzpicture}[line width = 1pt]
  
   \filldraw[color=lightgray] (0,0) -- (4,0) -- (4,2) -- (0,2) -- (0,0); 
   
   \filldraw[color=lightgray] (0,6) -- (4,6) -- (4,8) -- (0,8) -- (0,6); 
   
   \filldraw[color=lightgray] (0,2) to [out=90, in=-90] (2,4) -- (4,4) to [out=-90, in=90] (2,2); 
   
   \draw (0,2) to [out=90, in=-90] (2,4)   (2,2) to [out=90, in=-90] (4,4); 
   
   \filldraw[color=lightgray] (0,4) to [out=90, in=-90] (2,6) -- (4,6) to [out=-90, in=90] (2,4); 
   
   \filldraw[color=lightgray] (2,2) to [out=90, in=-90] (0,4) -- (2,4) to [out=-90, in=90] (4,2) -- (2,2); 
   
   \draw (0,4) to [out=90, in=-90] (2,6)   (2,4) to [out=90, in=-90] (4,6)
         (1,4) to [out=90, in=-90] (3,6)   (1.5,4) to [out=90, in=-90] (3.5,6); 
   
   \filldraw[color=lightgray] (2,4) to [out=90, in=-90] (0,6) -- (2,6) to [out=-90, in=90] (4,4) -- (2,4); 
   
   \draw (0,2) -- (0,0)   (4,0) -- (4,2)   (2,0) -- (2,2)   (2.5,0) -- (3,2)   (3,0) -- (3.5,2); \draw[very thin] (0,2) -- (4,2); 
   
   \draw (2,2) to [out=90, in=-90] (0,4)   (2,4) to [out=-90, in=90] (4,2)
         (3,2) to [out=90, in=-90] (1,4)   (1.5,4) to [out=-90, in=90] (3.5,2); 
   
   \draw (2,4) to [out=90, in=-90] (0,6)   (2,6) to [out=-90, in=90] (4,4); 
   
   \draw (0,6) -- (0,8) -- (4,8) -- (4,6)   (2,6) -- (2,8)   (3,6) -- (2.5,8)   (3.5,6) -- (3,8); \draw[very thin] (4,6) -- (0,6); 
   
   \draw (4,3.9) -- (4,4.1);
   
   \begin{scope}[yshift=2cm,yscale=-1]
    \filldraw[color=lightgray] (0,2) to [out=90, in=-90] (2,4) -- (4,4) to [out=-90, in=90] (2,2); 
   
    \draw (0,2) to [out=90, in=-90] (2,4)   (2,2) to [out=90, in=-90] (4,4); 
   
    \filldraw[color=lightgray] (0,4) to [out=90, in=-90] (2,6) -- (4,6) to [out=-90, in=90] (2,4); 
   
    \filldraw[color=lightgray] (2,2) to [out=90, in=-90] (0,4) -- (2,4) to [out=-90, in=90] (4,2) -- (2,2); 
   
    \draw (0,4) to [out=90, in=-90] (2,6)   (2,4) to [out=90, in=-90] (4,6)
         (0.5,4) to [out=90, in=-90] (2.5,6)   (1,4) to [out=90, in=-90] (3,6); 
   
    \filldraw[color=lightgray] (2,4) to [out=90, in=-90] (0,6) -- (2,6) to [out=-90, in=90] (4,4) -- (2,4); 
   
    \draw (2,2) to [out=90, in=-90] (0,4)   (2,4) to [out=-90, in=90] (4,2)
          (2.5,2) to [out=90, in=-90] (0.5,4)   (1,4) to [out=-90, in=90] (3,2); 
   
    \draw (2,4) to [out=90, in=-90] (0,6)   (2,6) to [out=-90, in=90] (4,4); 
   
    \draw (4,3.9) -- (4,4.1)   (0,6) -- (4,6); \draw[very thin] (0,2) -- (4,2);
   \end{scope}

   \begin{scope}[xshift=6cm]
    \filldraw[color=lightgray] (0,6) -- (2,6) -- (2,2) -- (0,2);
    
    \draw (0,6) -- (0,2)   (2,6) -- (2,2);
   
    \filldraw[color=lightgray] (0,0) -- (4,0) -- (4,2) -- (0,2) -- (0,0); 
   
    \filldraw[color=lightgray] (0,6) -- (4,6) -- (4,8) -- (0,8) -- (0,6); 
   
    \filldraw[color=lightgray] (2,2) to [out=90, in=-90] (3,4) -- (4,4) to [out=-90, in=90] (3,2); 
   
    \draw (2,2) to [out=90, in=-90] (3,4)   (3,2) to [out=90, in=-90] (4,4); 
   
    \filldraw[color=lightgray] (2,4) to [out=90, in=-90] (3,6) -- (4,6) to [out=-90, in=90] (3,4); 
   
    \filldraw[color=lightgray] (3,2) to [out=90, in=-90] (2,4) -- (3,4) to [out=-90, in=90] (4,2) -- (3,2); 
   
    \draw (2,4) to [out=90, in=-90] (3,6)   (3,4) to [out=90, in=-90] (4,6)
          (2.5,4) to [out=90, in=-90] (3.5,6); 
   
    \filldraw[color=lightgray] (3,4) to [out=90, in=-90] (2,6) -- (3,6) to [out=-90, in=90] (4,4) -- (3,4); 
   
    \draw (0,2) -- (0,0)   (4,0) -- (4,2)   (2,0) -- (2,2)   (2.5,0) -- (3,2)   (3,0) -- (3.5,2); \draw[very thin] (0,2) -- (4,2); 
   
    \draw (3,2) to [out=90, in=-90] (2,4)   (3,4) to [out=-90, in=90] (4,2)
          (2.5,4) to [out=-90, in=90] (3.5,2); 
   
    \draw (3,4) to [out=90, in=-90] (2,6)   (3,6) to [out=-90, in=90] (4,4); 
   
    \draw (0,6) -- (0,8) -- (4,8) -- (4,6)   (2,6) -- (2,8)   (3,6) -- (2.5,8)   (3.5,6) -- (3,8); \draw[very thin] (4,6) -- (0,6); 
   
    \draw (4,3.9) -- (4,4.1);
   \end{scope}

   \begin{scope}[xshift=6cm,yshift=2cm, yscale=-1]
    \filldraw[color=lightgray] (0,6) -- (2,6) -- (2,2) -- (0,2);
    
    \draw (0,6) -- (0,2)   (2,6) -- (2,2);
   
    \filldraw[color=lightgray] (2,2) to [out=90, in=-90] (3,4) -- (4,4) to [out=-90, in=90] (3,2); 
   
    \draw (2,2) to [out=90, in=-90] (3,4)   (3,2) to [out=90, in=-90] (4,4)
          (2.5,2) to [out=90, in=-90] (3.5,4); 
   
    \filldraw[color=lightgray] (2,4) to [out=90, in=-90] (3,6) -- (4,6) to [out=-90, in=90] (3,4); 
   
    \filldraw[color=lightgray] (3,2) to [out=90, in=-90] (2,4) -- (3,4) to [out=-90, in=90] (4,2) -- (3,2); 
   
    \draw (2,4) to [out=90, in=-90] (3,6)   (3,4) to [out=90, in=-90] (4,6); 
   
    \filldraw[color=lightgray] (3,4) to [out=90, in=-90] (2,6) -- (3,6) to [out=-90, in=90] (4,4) -- (3,4); 
   
    \draw (3,2) to [out=90, in=-90] (2,4)   (3,4) to [out=-90, in=90] (4,2); 
   
    \draw (3,4) to [out=90, in=-90] (2,6)   (3,6) to [out=-90, in=90] (4,4)
          (2.5,6) to [out=-90, in=90] (3.5,4); 
   
    \draw (4,3.9) -- (4,4.1)   (0,6) -- (4,6); \draw[very thin] (0,2) -- (4,2);
   \end{scope}

  \end{tikzpicture}
 \caption{The picture on the left is the commutator $[x_1,\beta_{1,2}]$. Since $X_{ess}(\beta_{1,2})=\{0,\frac{1}{2},1\}$, and $x_1$ fixes this set, this commutator is trivial, and one can check that the resulting ``ribbon diagram'' indeed represents the trivial element. The picture on the right demonstrates that $[x_1,\beta_{2,3}]$ is not trivial, as the ribbon diagram does not represent the trivial element, and indeed $X_{ess}(\beta_{2,3})=\{0,\frac{1}{2},\frac{3}{4},1\}$, which is not fixed by $x_1$.}
 \label{fig:commuting}
\end{figure}
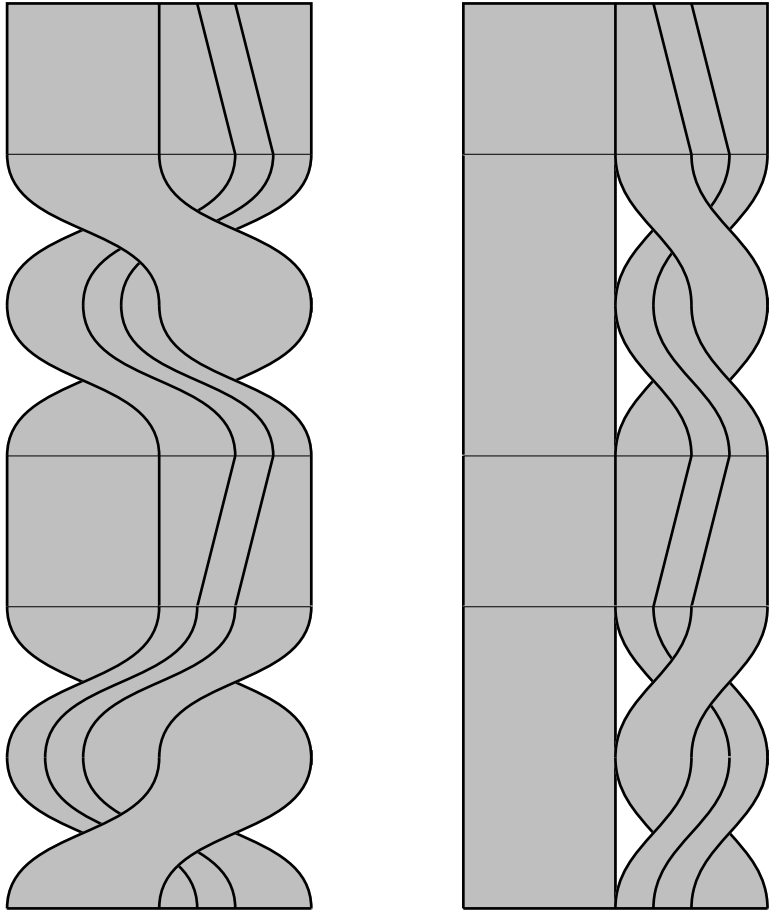

\begin{lemma}\label{lem:conjugator}
 Let $g\in\Pbr$ and $1\neq f\in F$. Then there exists $h\in F$ such that $[h,g]=1$ but $[h,f]\neq 1$.
\end{lemma}

\begin{proof}
 Order the elements of $X_{ess}(g)$ by $0=x_0<x_1<\cdots x_{r-1}<x_r=1$. By Proposition~\ref{prop:commuting}, it suffices to find $h\in F$ that fixes $X_{ess}(g)$ but does not commute with $f$.
 
 First suppose $f$ fixes $X_{ess}(g)$. Let $1\le i\le r$. Suppose that every $h$ with support in $[x_{i-1},x_i]$ commutes with $f$. Then $f|_{[x_{i-1},x_i]}$ must be trivial. Since $f\neq 1$, this cannot happen for every $i$, so we conclude that there exists $h$ with support in $[x_{i-1},x_i]$ for some $i$ such that $[h,f]\neq 1$, and since the support of $h$ is disjoint from $X_{ess}(g)$, also $[h,g]=1$.
 
 Now suppose $f$ does not fix $X_{ess}(g)$. There exists an element $h\in F$ whose fixed point set is precisely $X_{ess}(g)$, so in particular $[h,g]=1$. If $f$ were to commute with $h$ then it would necessarily stabilize its fixed point set, so instead we conclude that $[h,f]\neq 1$.
\end{proof}

\begin{proof}[Proof of Theorem~\ref{thrm:alt}]
 Let $t\in\Fbr\setminus\Pbr$. Write $t=gf$ for $g\in\Pbr$ and $1\neq f\in F$. By Lemma~\ref{lem:conjugator} we can choose $h\in F$ such that $h$ commutes with $g$ but not $f$. In particular, the normal closure of $t$ contains $(f\I g\I)(g hfh\I) = f\I hfh\I$. This is a non-trivial element of $F$ (even of $[F,F]$), so by Proposition~\ref{prop:normal_closure_F}, the normal closure of $t$ contains $[\Fbr,\Fbr]$.
\end{proof}

\begin{remark}\label{rmk:subgroup_remarks}
 Since $\Pbr$ contains every pure braid group, there is no hope of classifying all subgroups of $\Fbr$ in any real sense. At least we do know that every finitely generated subgroup of $\Pbr$ must lie in some $PB_n$, since $\Pbr$ is a direct limit of copies of the $PB_n$. Another interesting fact is that, since $\Pbr$ contains every pure braid group, it also contains every right-angled Artin group, by a result of Kim and Koberda \cite{kim15}. In particular, $\Fbr$ and $\Vbr$ are examples of finitely presented (even type $\F_\infty$) groups that contain every right-angled Artin group.
\end{remark}

Now that we have an Alternative for $\Fbr$, we can derive one for $\Vbr$. The two options for a normal subgroup turn out to be that it is either contained in $\Pbr$, or else equals all of $\Vbr$. To prove this we will quote a result from \cite{bux08} that says that the normal closure of $[F,F]$ in $\Vbr$ is all of $\Vbr$. The proof is similar to the first part of the proof of our Proposition~\ref{prop:normal_closure_F}.

\begin{corollary}[Alternative for $\Vbr$]\label{cor:BV_alt}
 Let $N$ be a proper normal subgroup of $\Vbr$. Then $N$ is contained in $\Pbr$.
\end{corollary}

\begin{proof}
 Suppose $N$ is not contained in $\Pbr$, so by Theorem~\ref{thrm:alt} $N$ contains $[\Fbr,\Fbr]$. In particular $N$ contains $[F,F]$, and so by \cite[Lemma~20]{bux08}, $N=\Vbr$.
\end{proof}

The fact that $V$ is simple tells us that either $N\le\Pbr$ or $N\Pbr=\Vbr$, so the content of the corollary is that in the latter case in fact $N=\Vbr$.

\begin{corollary}
 $\Pbr$ is characteristic in $\Vbr$ and $\Fbr$.
\end{corollary}

\begin{proof}
 The statement for $\Vbr$ is immediate from Corollary~\ref{cor:BV_alt}. The statement for $\Fbr$ follows from Theorem~\ref{thrm:alt} once we observe that $[\Fbr,\Fbr]$ is characteristic.
\end{proof}

\section{Normal subgroups over $[\Fbr,\Fbr]$}\label{sec:abelain_quotients}

Theorem~\ref{thrm:alt} tells us that normal subgroups of $\Fbr$ either contain $[\Fbr,\Fbr]$, or else live in $\Pbr$. The latter situation is rather complicated, and we will discuss some examples in Section~\ref{sec:non_abelian_quotients}, along with some general results. Normal subgroups over $[\Fbr,\Fbr]$ are more tractable though. In this section we compute the Bieri--Neumann--Strebel invariant $\Sigma^1(\Fbr)$, which sheds some light on such subgroups, for instance by characterizing which of them are finitely generated.

\subsection{The BNS invariant}\label{sec:BNS}

The Bieri--Neumann--Strebel (BNS) invariant of a finitely generated group $G$, introduced in \cite{bieri87}, is a geometric invariant $\Sigma^1(G)$ that, among other things, provides a means of understanding normal subgroups of $G$ containing the commutator subgroup $[G,G]$. For instance $\Sigma^1(G)$ tells us when such a normal subgroup is finitely generated or not.

Historically, $\Sigma^1(G)$ has proved to be difficult to compute in general. Some groups for which $\Sigma^1$ is interesting and has been successfully computed include right-angled Artin groups \cite{meier95}, pure braid groups \cite{koban15}, pure loop braid groups \cite{orlandi-korner00} and Thompson's group $F$ \cite{bieri87,bieri10}.

The BNS-invariant $\Sigma^1(G)$ of a finitely generated group $G$ is defined as follows. Consider characters $\chi\colon G\to\R$ of $G$. Two characters $\chi$ and $\chi'$ are \emph{equivalent} if there exists $c\in\R^{>0}$ such that $\chi(g)=c\chi'(g)$ for all $g\in G$. The equivalence classes of non-trivial characters form the \emph{character sphere} $\Sigma(G)$ of $G$. It is a $d$-sphere if the torsion-free rank of $G/[G,G]$ is $d+1$. Now pick a finite generating set $S$ for $G$ and let $\Gamma(G,S)$ be the Cayley graph. For a character $\chi\colon G\to \R$ let $\Gamma(G,S)^{\chi\ge0}$ be the full subgraph of $\Gamma(G,S)$ spanned by those vertices $g\in G$ with $\chi(g)\ge0$. The BNS-invariant $\Sigma^1(G)$ is a subset of $\Sigma(G)$, which does not depend on $S$, defined by:
$$\Sigma^1(G)\defeq\{[\chi]\in \Sigma(G)\mid \Gamma(G,S)^{\chi\ge0} \text{ is connected}\}\text{.}$$

The following is one of the main applications of $\Sigma^1(G)$.

\begin{cit}\cite[Theorem~B1]{bieri87}\label{cit:fin_props}
 Let $G$ be a finitely generated group and let $N\triangleleft G$ with $G/N$ abelian. Then $N$ is finitely generated if and only if for every $[\chi]\in \Sigma(G)$ such that $\chi(N)=0$, we have $[\chi]\in\Sigma^1(G)$.
\end{cit}

Another important fact is that $\Sigma^1(G)$ is invariant under automorphisms of $G$. As is standard, we denote the complement $\Sigma(G)\setminus \Sigma^1(G)$ by $\Sigma^1(G)^c$.

\subsection{Tools}\label{sec:tools}

In this subsection we establish some terminology and notation, and cite some useful results that we will use to calculate $\Sigma^1(\Fbr)$ in the following section. First we collect some definitions. Let $G$ be a group, and let $I,J\subseteq G$. We say that $J$ \emph{dominates} $I$ if every element of $I$ commutes with some element of $J$. The \emph{commuting graph} $C(J)$ of $J$ is the graph with vertex set $J$ and an (unoriented) edge between $a$ and $b$ if $a$ and $b$ commute. We say that $g\in G$ \emph{survives} under a character $\chi$ if $\chi(g)\ne0$. Otherwise we say it \emph{dies}, or that $\chi$ \emph{kills} it. If $g$ survives under $\chi$ we will also sometimes call it \emph{$\chi$-hyperbolic}.

The point of all this terminology is a useful criterion to determine if a character is in $\Sigma^1(G)$:

\begin{cit}[Survivors dominating generators]\label{cit:conn_dom}\cite[Lemma~1.9]{koban15}
 Let $G$ be a group and $\chi$ a character of $G$. Suppose there are sets $I,J\subseteq G$ such that $I$ generates $G$, every element of $J$ survives under $\chi$, $J$ dominates $I$, and $C(J)$ is connected. Then $[\chi]\in\Sigma^1(G)$.
\end{cit}

We will also make use of the following standard result, cf.~\cite[Lemma~1.3]{koban15}:

\begin{cit}[Quotients]\label{cit:quotients}
 Let $\pi\colon G\onto H$ be an epimorphism of groups. Let $\chi$ be a character of $H$ and let $\phi \defeq \chi\circ\pi$ be the corresponding character of $G$. If $[\phi]\in\Sigma^1(G)$ then $[\chi]\in\Sigma^1(H)$.
\end{cit}

\subsection{The BNS invariant $\Sigma^1(\Fbr)$}\label{sec:BNS_BF}

The answer is:

\begin{theorem}\label{thrm:BNS_BF}
 The Bieri--Neumann--Strebel invariant $\Sigma^1(\Fbr)$ for $\Fbr$ consists of all points on the sphere $\Sigma(\Fbr)=S^3$ except for the points $[\phi_0]$ and $[\phi_1]$, where $\phi_0$ and $\phi_1$ are as defined in Section~\ref{sec:chars}.
\end{theorem}

We will prove the theorem by looking at various cases. First we take care of the points $[\pm\phi_i]$. Here we will appeal to symmetry under an automorphism $\rho \colon \Fbr \to \Fbr$ that switches the roles of $\phi_0$ and $\phi_1$. The automorphism $\rho$ takes $(T_-,p,T_+)$, viewed as a split-braid-merge diagram from top to bottom living in $3$-space, and rotates it 180 degrees about an axis passing through the roots of both trees.

\begin{observation}\label{obs:ad_hoc}
 We have $[\phi_i]\in\Sigma^1(\Fbr)^c$ and $[-\phi_i]\in\Sigma^1(\Fbr)$, for $i=0,1$.
\end{observation}

\begin{proof}
 First note that we need only check the statements for $\pm\phi_0$, since the automorphism $\rho$ switches the roles of $\phi_0$ and $\phi_1$.

 We know that $[\phi_0]\in\Sigma^1(\Fbr)^c$ by Citation~\ref{cit:quotients}, since we have an epimorphism $\Fbr\onto F$ and the induced character $[\chi_0]$ on $F$ is in $\Sigma^1(F)^c$ \cite{bieri10}. 
 
 For the other statement, recall from Lemma~\ref{lem:hnn} that $\Fbr$ is an ascending HNN-extension of $\Fbr(1)$ by $x_0$. We have $-\phi_0(\Fbr(1))=0$ and $-\phi_0(x_0)=1$. Moreover, $\Fbr(1)$ is finitely generated, by arguments similar to those in \cite{brady08} that show $\Fbr$ is finitely generated. Hence by \cite[Theorem~2.1(1)]{bieri10}, we have $[-\phi_0]\in\Sigma^1(\Fbr)$.
\end{proof}

Now that we have handled $[\pm\phi_i]$, the strategy for the remaining characters is as follows. First we look at characters of $\Fbr$ that do not kill $\Pbr$. We first suppose that $\chi$ has non-zero $\omega_1$ component. Next we suppose that $\chi$ does have zero $\omega_1$ component and non-zero $\omega_0$ component (see Section~\ref{sec:chars} for definitions). Then we consider the case when $\chi$ does kill $\Pbr$, but has non-zero $\phi_0$ and $\phi_1$ components. In each of these three cases, we are able to apply Citation~\ref{cit:conn_dom} to conclude that $[\chi]\in\Sigma^1(\Fbr)$. 
 
\subsection*{Case 1} First we assume that the $\omega_1$ component of $[\chi]$ is non-zero. We will use Citation~\ref{cit:conn_dom} to show that $[\chi]$ is guaranteed to be in $\Sigma^1(\Fbr)$. The main trick is that every $\alpha_{i,i+1}$ is $\chi$-hyperbolic.
 
\begin{lemma}[Non-zero $\omega_1$]\label{lem:omega1_non_0}
 Let $\chi$ be any character with non-zero $\omega_1$ component. Then $[\chi]\in\Sigma^1(\Fbr)$.
\end{lemma}
 
\begin{proof}
 Let
 $$J_1\defeq\{\alpha_{i,i+1}\mid i\ge 1\} \text{ and } I_1\defeq\{\alpha_{i,j}\mid 1\le i<j\} \cup \{\beta_{i,j}\mid 1\le i<j-2\} \cup \{x_2,x_2 x_0\I\} \text{.}$$
 We claim that $J_1$ dominates $I_1$, $C(J_1)$ is connected, $I_1$ generates $\Fbr$, and every element of $J_1$ survives under $\chi$. First note that any $\alpha_{i,j}$ commutes with $\alpha_{j+1,j+2}$ (B1), that any $\beta_{i,j}$ with $i<j-2$ commutes with $\alpha_{i+1,i+2}$ (B5), that $x_2$ commutes with $\alpha_{1,2}$ (D5), and that $x_2 x_0\I$ commutes with $\alpha_{4,5}$ (D1). This tells us that $J_1$ dominates $I_1$. Now observe that every element of $J_1$ commutes with $\alpha_{1,2}$ except for $\alpha_{2,3}$ (B1), but $\alpha_{2,3}$ commutes with $\alpha_{4,5}$, so $C(J_1)$ is connected. That $I_1$ generates $\Fbr$ is routine to check in light of (C), since this ensures that $\beta_{i,j}$ with $i<j\le i+2$ can be obtained using $I_1$, and since $x_2$ and $x_2 x_0\I$ generate $F$. Finally, every element of $J_1$ survives under $\omega_1$ and dies under $\phi_0$, $\phi_1$ and $\omega_0$, so necessarily survives under $\chi$.
\end{proof}
 
\subsection*{Case 2} Next suppose that the $\omega_0$ component is non-zero. We will again use Citation~\ref{cit:conn_dom}, but with a different dominating set $J_0$ and generating set $I_0$.
 
First we need to discuss central elements of $PB_n$. For any $n$, the center $Z(PB_n)$ is cyclic, generated by an element $\Delta_n$ that can be visualized as spinning the $n$ strands around in lockstep by 360 degrees. For any tree $T$ with $n$ leaves, the element
$$\delta(T)\defeq (T,\Delta_n,T)$$
commutes with every element of the form $(T,p,T)$ for $p\in PB_n$. However, it is important that we use the same tree $T$; there is no guarantee that $\delta(T)$ will commute with an element of the form $(T',p,T')$ if $T'$ is not $T$.
 
Another important observation about the elements $\delta(T)$ is that they survive under $\omega_0$; indeed, $\omega_0$ of such an element is $1$. As a remark, $\omega_1$ of such an element equals $n-1$, but this will not matter in what follows.

\begin{lemma}\label{lem:comm_dom}
 Any element of $[F,F]\le \Fbr$ commutes with some conjugate of an element of the form $\beta_{1,j}$.
\end{lemma}

\begin{proof}
 Conjugates of $\beta_{1,j}$ by elements of $F$ amount to subdividing $[0,1]$ into $j$ subintervals and then braiding the first and last ones around each other. In particular, the essential endpoints of the subdivision are $0$, the first endpoint after $0$, the last endpoint before $1$, and $1$; see the definition before Observation~\ref{obs:essential_invariant}. Given an element $f\in [F,F]$, so the support of $f$ in $[0,1]$ is bounded away from $0$ and $1$, we can choose a subdivision in which these four essential points are disjoint from the support of $f$. Then $f$ will commute with the conjugate of the $\beta_{1,j}$ corresponding to this subdivision by Proposition~\ref{prop:commuting}.
\end{proof}

\begin{proposition}[Non-zero $\omega_0$]\label{prop:omega0_non_0}
 Let $\chi$ be any character with non-zero $\omega_0$ component. Then $[\chi]\in\Sigma^1(\Fbr)$.
\end{proposition}
 
\begin{proof}
 Thanks to Lemma~\ref{lem:omega1_non_0} we may assume that $\chi$ has $\omega_1$ component zero. In particular, every $\delta(T)$ is $\chi$-hyperbolic.
  
 Let $J_0$ be the set of all conjugates of $\beta_{1,j}$ for $j\ge 2$ and all $\delta(T)$ for all trees $T$ and all $n\ge1$. Then every element of $J_0$ survives under $\omega_0$ and dies under $\phi_0$ and $\phi_1$, and since $\chi$ has $\omega_1$ component zero this tells us that every element of $J_0$ survives under $\chi$. We next claim that $C(J_0)$ is connected, and in fact that it is connected with diameter $2$. Indeed, given any two elements $x,y$ of $J_0$ there exists a tree $T$ with $n$ leaves such that $x=(T,p,T)$ and $y=(T,q,T)$ for $p,q\in PB_n$. Then $x$ and $y$ both commute with $\delta(T)$.
  
 Now we need to find a generating set $I_0$ for $\Fbr$ that is dominated by $J_0$. Since every element of $\Pbr$ commutes with some $\delta(T)$, we may as well include all of $\Pbr$ in $I_0$. We just need to add elements to $I_0$ that are dominated by $J_0$ until we have generated $F\le\Fbr$. By Lemma~\ref{lem:comm_dom}, we can add all of $[F,F]$ to $I_0$. Also, note that $\beta_{1,2}$ commutes with any element of $F$ fixing $1/2$, by Proposition~\ref{prop:commuting}, and $F$ is generated by such elements together with $[F,F]$, so we are done.
\end{proof}

\subsection*{Case 3} Now suppose that $\chi$ kills $\Pbr$, so its $\omega_0$ and $\omega_1$ components are both zero. Also assume, for this case, that the $\phi_0$ and $\phi_1$ components of $\chi$ are not both zero. We will find yet another pair of sets $J_F$ and $I_F$ such that Citation~\ref{cit:conn_dom} applies. As a remark, the proof recovers the fact that the restrictions of such characters to $F$ are in $\Sigma^1(F)$, originally proved in \cite{bieri87}.

\begin{lemma}\label{lem:most_of_F}
 Let $\chi$ be a character that kills $\Pbr$ and whose $\phi_0$ and $\phi_1$ components are both non-zero. Then $[\chi]\in\Sigma^1(\Fbr)$.
\end{lemma}

\begin{proof}
 Let $J_F$ be the set of all elements of $F\le\Fbr$ whose support has precisely one of the endpoints of $[0,1]$ as a limit point. Since elements with disjoint supports commute, it is straightforward to verify that $C(J_F)$ is connected. Also, any element of $J_F$ survives under $\chi$ by our hypothesis on $\chi$.
 
 Now define $I_F'$ to be $[F,F]\cup J_F$. It is straightforward to check that $J_F$ dominates $I_F'$ and that $I_F'$ generates $F$, so we recover the fact that $[\chi|_F]\in\Sigma^1(F)$. Now let $I_F$ be the union of $I_F'$ with the set of elements of the form $\alpha_{i,j}$ and $\beta_{i,j}$ ($1\le i<j$), so $I_F$ generates $\Fbr$. Any $\alpha_{i,j}$ or $\beta_{i,j}$ commutes with $x_j$ by (D5) and (D9), which is in $J_F$, so $J_F$ dominates $I_F$.
\end{proof}
 
We can now put the cases together and compute $\Sigma^1(\Fbr)$.
 
\begin{proof}[Proof of Theorem~\ref{thrm:BNS_BF}]
 Let $[\chi]\in \Sigma(\Fbr)$, say
 $$\chi=a\phi_0 + b\phi_1 + c\omega_0 + d\omega_1 \text{.}$$
 If $c\ne0$ or $d\ne0$ then $[\chi]\in\Sigma^1(\Fbr)$ by Lemma~\ref{lem:omega1_non_0} and Proposition~\ref{prop:omega0_non_0}, so assume $c=0$ and $d=0$. If $a$ and $b$ are both non-zero then $[\chi]\in\Sigma^1(\Fbr)$ by Lemma~\ref{lem:most_of_F}. The four remaining points of $\Sigma(\Fbr)$ are $[\pm\phi_i]$ for $i=0,1$, which are handled by Observation~\ref{obs:ad_hoc}.
\end{proof}

An immediate application is that we know exactly when normal subgroups of $\Fbr$ containing $[\Fbr,\Fbr]$ are finitely generated.

\begin{corollary}\label{cor:fin_gen_normal_subgroups}
 Let $N$ be a normal subgroup of $\Fbr$, and suppose that $N$ contains $[\Fbr,\Fbr]$. Then $N$ is finitely generated if and only if $N\not\subseteq\ker(\phi_0)$ and $N\not\subseteq\ker(\phi_1)$.
\end{corollary}

\begin{proof}
 First note that by assumption $\Fbr/N$ is abelian. If $N$ is contained in the kernel of either $\phi_0$ or $\phi_1$, then $N$ is not finitely generated, by Citation~\ref{cit:fin_props}. Now suppose that $N$ is not contained in either kernel. Let $\chi$ be any non-trivial character of $\Fbr$ such that $\chi(N)=0$, so $[\chi]\not\in\{[\phi_0],[\phi_1]\}$. In particular, $[\chi]\in\Sigma^1(\Fbr)$, so $N$ is finitely generated by Citation~\ref{cit:fin_props}.
\end{proof}

We can combine this result with a proposition from the previous section:

\begin{corollary}\label{cor:closure_of_F}
 Let $f,g\in F$ be elements such that $\chi_0(f)\neq 0$, $\chi_1(f)=0$, $\chi_0(g)=0$ and $\chi_1(g)\neq 0$. Then the normal closure $\langle\!\langle f,g\rangle\!\rangle$ in $\Fbr$ is finitely generated.
\end{corollary}

\begin{proof}
 The normal closure contains the commutator subgroup by Proposition~\ref{prop:normal_closure_F}. Hence it is finitely generated by Corollary~\ref{cor:fin_gen_normal_subgroups}.
\end{proof}

\section{Normal subgroups under $\Pbr$}\label{sec:non_abelian_quotients}

Before discussing normal subgroups of $\Vbr$ and $\Fbr$ contained in $\Pbr$, we should generally inspect the subgroups of $\Pbr$. Since $\Pbr$ is a direct limit of copies of pure braid groups, we know that every subgroup of $\Pbr$ is a direct limit of subgroups of the $PB_n$, so our first goal is to pin down what can happen. For each $n\in\N$ and each tree $T$ with $n$ leaves, recall from Section~\ref{sec:Pbr} that $PB_T$ denotes the copy of $PB_n$ consisting of triples of the form $(T,p,T)$, and $\Pbr$ is the direct limit of the $PB_T$.

Given a family of subgroups $G_T$ of the $PB_T$, one for each tree $T$, we can consider the subgroup of $\Pbr$ generated by all the $G_T$. If we want any hope of recovering the family from the subgroup it generates, and hence of classifying the subgroups of $\Pbr$, we need some conditions on the family. The criteria to check this are as follows. Let $(G_T)_T$ be a family of subgroups with $G_T\le PB_T$ for each $T$. We will call the family \emph{coherent} if whenever $T\le T'$, the inclusion $PB_T\to PB_{T'}$ restricts to an inclusion $G_T\to G_{T'}$, i.e., $G_T\le G_{T'}$ as subgroups of $\Pbr$. If moreover $G_{T'}\cap PB_T$ equals $G_T$ we will call the family \emph{complete}; the condition here that is not immediate is $G_{T'}\cap PB_T \subseteq G_T$. The point is that we can recover a complete coherent family of subgroups of the $PB_T$ from the subgroup they generate in $\Pbr$, as the next proposition makes precise.

\begin{proposition}[Subgroups of $\Pbr$]\label{prop:PBV_subgroups}
 The subgroups of $\Pbr$ are in one-to-one correspondence with the complete coherent families of subgroups of the $PB_T$.
\end{proposition}

\begin{proof}
 Every coherent family yields a subgroup of $\Pbr$, namely the subgroup generated by the subgroups in the family. Also, given a subgroup $G$ of $\Pbr$, the family $(PB_T\cap G)_T$ is coherent and complete for trivial reasons. The only thing to check then is that two distinct complete coherent families yield distinct subgroups of $\Pbr$. Let $(G_T)_T$ be coherent and complete, and let $G$ be the subgroup of $\Pbr$ generated by the $G_T$. We claim that $PB_T\cap G \subseteq G_T$ for all $T$, after which we will be done, since the reverse inclusion is immediate.
 
 Let $g\in PB_T\cap G$. Since $G$ is generated by the $G_T$, we can write $g$ as a product $g = g_1\cdots g_r$ where each $g_i$ lies in some $G_{T_i}$. Let $T'$ be a common upper bound for $\{T\}\cup\{T_i\}_{i=1}^r$, so $g$ and all the $g_i$ lie in $PB_{T'}$. Since the family is coherent all the $g_i$ even lie in $G_{T'}$. This implies $g\in G_{T'}$, and now since the family is complete and $g\in PB_T$, we conclude that $g\in G_T$.
\end{proof}

A consequence of the proof is that given any subgroup $G\le \Pbr$, the unique complete coherent family that generates $G$ is $(PB_T\cap G)_T$.

For $n(T)$ the number of leaves of $T$, denote by
$$\psi_T \colon PB_T \to PB_{n(T)}$$
the isomorphism $(T,p,T)\mapsto p$.

\begin{lemma}[Families of normal subgroups]\label{lem:seq_nml_subgps}
 Let $(G_T)_T$ be a complete coherent family of subgroups $G_T\le PB_T$, and let $G\le \Pbr$ be the subgroup generated by the $G_T$. Then $G$ is normal in $\Fbr$ if and only if each $G_T$ is normal in $PB_T$, and for every $T$ and $S$ with $n(T)=n(S)$ we have $\psi_T(G_T)=\psi_S(G_S)$ in $PB_n$.
\end{lemma}

\begin{proof}
 First suppose each $G_T$ is normal in $PB_T$ and $\psi_T(G_T)=\psi_S(G_S)$ for all $T,S$ with $n(T)=n(S)$. An element of $G$ is a triple $g=(T,p,T)$ for some $T$ and $p\in \psi_T(G_T)$. Let $h=(S,q,U)$ be an arbitrary element of $\Fbr$. Since the family $(G_T)_T$ is coherent, we can expand $T$, $S$ and $U$ until without loss of generality $h=(S,q,T)$, so $hgh^{-1}=(S,qpq^{-1},S)$. Now, $G_T$ is normal in $PB_T$, so $qpq^{-1} \in \psi_T(G_T)$, and so by hypothesis is also in $\psi_S(G_S)$, which means that $(S,qpq^{-1},S) \in G_S$. We conclude that $hgh^{-1}\in G$.
 
 Now suppose $G$ is normal in $\Fbr$. It is immediate that $G_T$ is normal in $PB_T$. Let $T$ and $S$ both have $n$ leaves, and let $(T,p,T)\in G_T$, so $(T,p,T)\in G$. Since $G$ is normal we also have $(S,p,S)\in G$. But $G_S=PB_S\cap G$, so $p\in \psi_S(G_S)$. This shows $\psi_T(G_T)\subseteq \psi_S(G_S)$, and the reverse inclusion follows by the same argument.
\end{proof}

In conclusion, the normal subgroups of $\Fbr$ contained in $\Pbr$ are obtained precisely by choosing a normal subgroup $G_n\triangleleft PB_n$ for each $n$ such that for every $1\le k\le n$, we have:
\begin{equation}\label{eqn:complete_coherent}
 \clone_k^n(PB_n)\cap G_{n+1}=\clone_k^n(G_n)\text{.}
\end{equation}
Indeed, this equation ensures that the family $(G_T)_T$ given by $G_T\defeq G_{n(T)}$ is complete and coherent. We will call $(G_n)_{n\in\N}$ a \emph{complete coherent sequence of normal subgroups}.

Given such a sequence $(G_n)_{n\in\N}$, we will denote the corresponding subgroup of $\Pbr$ by $\Thkern{G_*}$ (following the notation in \cite{witzel18}), so $\Thkern{G_*} \triangleleft \Fbr$. In terms of triples, we have:
$$\Thkern{G_*} = \{(T,p,T)\mid T \text{ has } n \text{ leaves and } p\in G_n\} \text{.}$$

It is straightforward to decide when $\Thkern{G_*}$ is even normal in $\Vbr$.

\begin{lemma}\label{lem:Vbr_normal}
 Let $(G_n)_{n\in\N}$ be a sequence satisfying Equation~\ref{eqn:complete_coherent}, so $\Thkern{G_*}$ is normal in $\Fbr$. Then $\Thkern{G_*}$ is normal in $\Vbr$ if and only if each $G_n\le PB_n$ is normal in $B_n$.
\end{lemma}

\begin{proof}
 First suppose $\Thkern{G_*}$ is normal in $\Vbr$. Let $T$ be arbitrary, and let $p\in G_n$. For any $b\in B_n$, we have
 $$(T,bpb^{-1},T) = (T,b,T)(T,p,T)(T,b^{-1},T) \in \Thkern{G_*}$$
 and $(T,bpb^{-1},T) \in PB_T$, so in fact $(T,bpb^{-1},T) \in G_T$, which implies that $bpb^{-1} \in G_n$.
 
 Now suppose that $G_n$ is normal in $B_n$ for all $n$. Let $g\in \Thkern{G_*}$ and $h\in\Vbr$. Choose $T$, $S$, $p\in PB_{n(T)}$ and $b\in B_{n(T)}$ such that $g=(T,p,T)$ and $h=(S,b,T)$; in other words expand the triples until they have a common bottom tree. Now
 $$hgh^{-1} = (S,b,T)(T,p,T)(T,b^{-1},S) = (S,bpb^{-1},S)$$
 is in $G_S$ since $bpb^{-1} \in G_{n(S)}$. We conclude that $\Thkern{G_*}$ is normal in $\Vbr$.
\end{proof}

Here is one family of examples of complete coherent sequences  of normal subgroups $(G_n)_{n\in\N}$. Let $n\in\N$ and $1\le m\le n$. Call a pure braid $p\in PB_n$ \emph{$m$-loose} if it becomes trivial upon deleting all but any $m$ strands. For example, every pure braid is $1$-loose, and a pure braid $p$ is $2$-loose if and only if every total winding number between two strands is zero, if and only if $p\in[PB_n,PB_n]$. Let $\loose{n}{m}$ be the subgroup of $PB_n$ consisting of all $m$-loose braids, so $\loose{n}{m}$ is normal in $B_n$. We have:
$$PB_n=\loose{n}{1}>\loose{n}{2}=[PB_n,PB_n]>\loose{n}{3}>\cdots> \loose{n}{n-1}>\loose{n}{n}=\{1\}\text{.}$$
As a remark, $\loose{n}{n-1}$ is the group of \emph{Brunnian braids}, i.e., braids that become trivial upon removing any single strand.

\begin{lemma}\label{lem:clone_and_forget}
 For $1\le k\le n$ and $1\le m\le n$, we have $\clone_k^n(PB_n)\cap \loose{n+1}{m} = \clone_k^n(\loose{n}{m})$, i.e., Equation~\ref{eqn:complete_coherent} is satisfied.
\end{lemma}

\begin{proof}
 One direction is trivial: if $p\in \loose{n}{m}$, then $\clone_k^n(p) \in \loose{n+1}{m}$. Now suppose that $\clone_k^n(p) \in \loose{n+1}{m}$. Let $1\le i_1<i_2<\cdots <i_m\le n$ be the numbering of $m$ arbitrary strands of $p$, and let $S\defeq\{i_1,\dots,i_m\}$. Let $\pi_S\colon PB_n \to PB_m$ be the map that deletes all those strands not numbered by elements of $S$. Define
 $$S(k)\defeq\{i_1+\varepsilon_1,\dots,i_m+\varepsilon_m\}\text{,}$$
 where $\varepsilon_j$ is $0$ if $i_j\le k$ and is $1$ if $k<i_j$. Then $\pi_{S(k)}\circ \clone_k^n = \pi_S$. By assumption, $\pi_{S(k)}(\clone_k^n(p))=1$, and so $\pi_S(p)=1$. We conclude that $p\in \loose{n}{m}$.
\end{proof}

We emphasize that $m$-looseness is about deleting \emph{all but} any $m$ strands to get the trivial braid. If we instead considered deleting any $m$ strands to get the trivial braid, then this would \emph{not} give a coherent sequence. For instance, if $1\neq p\in \loose{n}{n-1}$ (so $p$ is Brunnian), then the lemma says $\clone_k^n(p)$ is in $\loose{n+1}{n-1}$, but it is not in $\loose{n+1}{n}$ (so not Brunnian), since deleting one of the cloned strands will bring us back to $p$, not to $1$.

We now have a concrete family of normal subgroups of $\Vbr$ contained in $\Pbr$, namely:

$$\thloose{m} \defeq \{(T,p,T)\mid T \text{ has } n \text{ leaves and } p\in \loose{n}{m}\}\text{.}$$

We have $\thloose{1}=\Pbr$ and $\thloose{2}=[\Pbr,\Pbr]$. As $m$ grows, we find a descending chain of normal subgroups
$$\cdots \triangleleft \thloose{3} \triangleleft \thloose{2}=[\Pbr,\Pbr] \triangleleft \Pbr=\thloose{1} \triangleleft \Fbr$$
with
$$\bigcap\limits_{m\in\N}\thloose{m} = \{1\}\text{.}$$

As a non-example of a complete coherent sequence, consider the sequence of normal subgroups of $B_n$ given by the centers $Z(B_n)\le PB_n$, generated by $\Delta_n$. Upon cloning, $\clone_k^n(\Delta_n) \le PB_{n+1}$ is not contained in $Z(B_{n+1})$, so this sequence is not coherent. Indeed, these subgroups, when considered as subgroups of $\Vbr$, normally generate all of $\Pbr$; in fact the single element $\beta_{1,2}=(R_2,\Delta_2,R_2)$ already normally generates $\Pbr$ in $\Vbr$.

When thinking of normal subgroups of (pure) braid groups, an obvious question is whether the coherent sequence $(PB_n^{(m)})_{n\in\N}$ of $m\thh$ derived subgroups is complete for fixed $m>2$. When $m=2$ we have $PB_n^{(2)}=\loose{n}{2}$, so the answer is yes, but the $m=3$ case is already unclear. Concretely, if $g$ is a product of commutators of products of commutators, and $g$ happens to feature a cloned strand, so $g=\clone_k(h)$ for some $h$, then is $h$ a product of commutators of products of commutators? All of these questions hold as well for the sequence of $m\thh$ terms of upper or lower central series, for fixed $m$, and for all the corresponding versions for the braid groups $B_n$.

\subsection{Quotients}\label{sec:quotients}

Given a complete coherent sequence of normal subgroups $(G_n)_{n\in\N}$, with limit $\Thkern{G_*}$, we can consider the quotients $\Fbr/\Thkern{G_*}$ and $\Vbr/\Thkern{G_*}$, which are somewhat straightforward to describe. The quotient map $\pi \colon \Vbr \onto \Vbr/\Thkern{G_*}$ takes a triple $(T,b,S)$ to a triple $(T,bG_n,S)$, where $n=n(T)=n(S)$. In particular the quotient can be described as the set of such triples, up to reduction and expansion, which are well defined since $(G_n)_{n\in\N}$ is complete and coherent. In the future we believe these quotients could be further inspected using the ``cloning systems'' framework from \cite{witzel18}.

\begin{example}[$\Fbr/\thloose{2}$]\label{ex:Thomp(H1(PB_*))}
 Note that since $\loose{n}{2}=[PB_n,PB_n]$, we have $PB_n/\loose{n}{2} = H_1(PB_n) = \Z^{\binom{n}{2}}$. Heuristically an element $\vec{v} \in H_1(PB_n)$ is a record of the total winding numbers of each pair of strands (hence the $\binom{n}{2}$) of a representative of $\vec{v}$ in $PB_n$. Fix a basis $(e_{i,j}\mid 1\le i<j\le n)$ for $\Z^{\binom{n}{2}}$. If $\vec{v}$ is the image of $p\in PB_n$ in the abelianization $H_1(PB_n)$, then the coefficient of $e_{i,j}$ in $\vec{v}$ is the total winding number of strands $i$ and $j$ in $p$.
 
 The quotient $\Fbr/\thloose{2}$ is described as follows. An element of $\Fbr/\thloose{2}$ is represented by a triple $(T,\vec{v},S)$ where $T$ and $S$ are trees with $n$ leaves and $\vec{v}\in H^1(PB_n)$. We consider such triples up to reduction and expansion, as in $\Fbr$; now expansion is described as follows. If we expand the $k\thh$ leaf of $T$ by attaching a caret, call the new tree $T'$, then we correspondingly replace $\vec{v}$ with an element $\clone_k^n(\vec{v})$ in $\Z^{\binom{n+1}{2}}$; the map $\clone_k^n$ is defined on the basis vectors as follows:
\begin{align*}
\clone_k^n(e_{i,j}) \defeq \left\{\begin{array}{ll}
                                    e_{i+1,j+1} &\text{if } 1\le k<i<j \\
                                    e_{i+1,j+1} + e_{i,j+1} &\text{if } 1\le k=i<j \\
                                    e_{i,j+1} &\text{if } 1\le i<k<j \\
                                    e_{i,j+1} + e_{i,j} &\text{if } 1\le i<k=j \\
                                    e_{i,j} &\text{if } 1\le i<j<k \text{.}
                                   \end{array}\right.
\end{align*}
 This should be compared to the relations (D1)--(D9) in the presentation of $\Fbr$ from Section~\ref{sec:presentations}, which also specify how to write $\clone_k^n(\alpha_{i,j})$ and $\clone_k^n(\beta_{i,j})$ as products of generators, e.g., $\clone_k^n(\alpha_{i,j}) = \alpha_{i+1,j+1}\alpha_{i,j+1}$ if $k=i$, and so forth.
\end{example}

One can show that, since no $H^1(PB_n)$ contains a non-abelian free group (being abelian), neither does $\Fbr/\thloose{2}$. In particular $\Fbr/\thloose{2}$ is not isomorphic to $\Fbr$. We take this as evidence that none of the $\Fbr/\thloose{m}$ should be isomorphic to $\Fbr$. On the other hand, the $m=2$ case is somewhat unique; for $m>2$, $PB_n/\loose{n}{m}$ does contain non-abelian free groups. Since $PB_n/\loose{n}{m}$ embeds into $\Fbr/\thloose{m}$ for any $n$, this is a proper quotient of $\Fbr$ that contains $F$ and contains non-abelian free groups.

\begin{question}\label{quest:mod_loose_iso}
 Are any of the quotients $\Fbr/\thloose{m}$ (for $m>2$) isomorphic to $\Fbr$ itself? Are any of the quotients $\Vbr/\thloose{m}$ (for $m\ge2$) isomorphic to $\Vbr$?
\end{question}

Note that for $\Vbr$ the above question includes the $m=2$ case, since $\Vbr/\thloose{2}$ does contain free subgroups (by virtue of $V$, unlike $F$, containing free subgroups). More generally, one can ask:

\begin{question}\label{quest:BF_Hopfian}
 Are $\Fbr$ and/or $\Vbr$ Hopfian?
\end{question}

This question becomes especially intriguing in light of the following:

\begin{proposition}\label{prop:PBV_not_Hopfian}
 The group $\Pbr$ is not Hopfian.
\end{proposition}

\begin{proof}
 For a non-trivial tree $T$ with $n$ leaves, define $T_L$ to be the subtree of $T$ whose root is the left child of the root of $T$. Let $T_R$ be the tree whose root is the right child of the root of $T$. For $0\le m\le n$, we have an epimorphism
 $$\phi_{n,m} \colon PB_n \to PB_m$$
 given by forgetting all the strands of a pure braid except for the first $m$ of them. Now define a map
 $$\phi_L \colon \Pbr \to \Pbr$$
 sending $(T,p,T)$ to $(T_L,\phi_{n,n(T_L)}(p),T_L)$. It is straightforward to check that this is well defined under reduction and expansion, and is a surjective homomorphism. It is also not injective; indeed the kernel contains every generator $\alpha_{i,j}$ and $\beta_{i,j}$ for $1\le i<j$, since these were defined using all-right trees.
\end{proof}

The normal subgroups $\ker(\phi_{n,m})$ do not form a coherent sequence, and in fact once they are considered inside $\Fbr$ or $\Vbr$, they normally generate all of $\Pbr$. This can be seen by noting that we catch every $\alpha_{i,j}$ and $\beta_{i,j}$ for $1\le i<j$, and the conjugates of these elements in $\Fbr$ generate all of $\Pbr$. Hence this does not give any direct hints about Question~\ref{quest:BF_Hopfian}.

\bibliographystyle{alpha}

\newcommand{\etalchar}[1]{$^{#1}$}

\end{document}